\documentclass[12pt]{amsart}
\usepackage[english]{babel}
\usepackage[cp1250]{inputenc}
\usepackage{graphicx}
\usepackage{url}
\usepackage{amsmath}
\usepackage{amsthm}
\usepackage{amsfonts}
\usepackage[all]{xy}
\usepackage{multicol}
\usepackage{hyperref}
\usepackage{pinlabel}
\usepackage{geometry}

\allowdisplaybreaks

\newtheorem{theorem}{Theorem}[section]
\newtheorem{lemma}[theorem]{Lemma}
\newtheorem{proposition}[theorem]{Proposition}
\newtheorem{corollary}[theorem]{Corollary}

\newcommand{\ZZ}{\mathbb {Z}}

\newcommand{\CC}{\mathbb {C}}
\newcommand{\TT}{\mathbb {T}}
\newcommand{\y}{Y_{m,n}}
\newcommand{\n}{N_{m,n}}
\newcommand{\m}{M_{m,n}}

\newcommand{\FF}{\mathbb{F}}
\newcommand{\ta}{\mathcal{T}^{+}}
\newcommand{\gr}{\widetilde{\operatorname{gr}}}
\def \sp {\operatorname{Spin}^{c}\nolimits }

\begin{document}
\title{Double plumbings of disk bundles over spheres}
\author{Eva Horvat}
\address{Faculty of Education\\
University of Ljubljana\\
Kardeljeva plo\v s\v cad 16\\
1000 Ljubljana, Slovenia}
\email{eva.berdajs@pef.uni-lj.si}
\date{July 18, 2013}
\begin{abstract}
We consider double plumbings of two disk bundles over spheres. We calculate the Heegaard--Floer homology with its absolute grading of the boundary of such a plumbing. Given a closed smooth 4--manifold $X$ and a suitable pair of classes in $H_{2}(X)$, we investigate when this pair of classes may be represented by a configuration of surfaces in $X$ whose regular neighborhood is a double plumbing of disk bundles over spheres. Using similar methods we study single plumbings of two disk bundles over spheres inside $X$.
\end{abstract}
\maketitle

\begin{section}{Introduction}
Given a smooth closed connected 4--manifold $X$ and a finite set of classes $C\subset H_{2}(X)$, an important question is what is the simplest configuration of surfaces in $X$ representing $C$. By simple we mean that each class should be represented by a surface of low genus and that the surfaces should have a low number of geometric intersections. Since it is always possible to remove cancelling pairs of intersection points by increasing the genus of a surface, both properties should be taken into account. Considerable work has been done to investigate the minimal genus of a given class in $H_{2}(X)$, first by proving the Thom conjecture (Kronheimer--Mrowka \cite{KM}) and then its generalizations by Morgan--Szab\'o--Taubes \cite{MST} and Ozsv\'ath--Szab\'o \cite{OS6}. When considering a con\-fi\-gu\-ra\-tion of surfaces, the sum of their genera is closely related to the number of their geometric intersections, as shown by Gilmer \cite{GILMER}. He showed that the the minimal number of such intersections can be estimated using the Casson--Gordon invariant. This estimate has been improved by Strle \cite{SASO} for configurations of $n=b_{2}^{+}(X)$ algebraically disjoint surfaces of positive self-intersection by an application of the Seiberg--Witten equations on a cylindrical end manifold. 

Multiple plumbings of two trivial disk bundles over spheres have been investigated by Sunukjian in his thesis \cite{NS}. He calculated the Heegaard-Floer homology of the boundary of such plumbings in cases where the two spheres are plumbed either once or zero times algebraically and $n$ times geometrically. 

We investigate the double plumbing $\n $ of two disk bundles with Euler classes $m$ and $n$ over spheres, which represents the simplest case of a configuration of two surfaces with algebraic (and geometric) intersection 2. We calculate the $d_{b}$--invariants of the Heegaard--Floer homology of $\partial \n $ \cite{OS2} and use an obstruction theorem \cite[Theorem 9.15]{OS4} to see when $\n $ can indeed be realized inside a given 4--manifold $X$ with $b_{2}^{+}(X)=2$. 

Denote by $\y $ the boundary of $\n $. For two integers $i$ and $j$, denote by $\mathfrak{t}_{i,j}$ the unique $\sp $ structure on $\n $ for which 
\begin{align}
\label{eqsp}
\langle c_{1}(\mathfrak{t}_{i,j}),s_{1}\rangle +m=2i
\end{align}
\begin{align}
\label{eqsp1}
\langle c_{1}(\mathfrak{t}_{i,j}),s_{2}\rangle +n=2j\;,
\end{align}   
where $s_{1},s_{2}\in H_{2}(\n )$ are the homology classes of the base spheres in the double plumbing. Let $\mathfrak{s}_{i,j}=\mathfrak{t}_{i,j}|_{\y }$. 

Throughout the paper, we denote by $\FF $ the field $\ZZ _{2}$ and by $\mathcal{T}^{+}$ the quotient module $\FF [U,U^{-1}]/U\FF [U]$. Our main result is the following:
\begin{theorem}
\label{th1}
Let $Y=\y $ be the boundary of a double plumbing of two disk bundles over spheres with Euler numbers $m$ and $n$, where $m,n\geq 4$. The Heegaard--Floer homology $HF^{+}(Y,\mathfrak{s})$ with $\FF $ coefficients is given by 
\begin{eqnarray*}
& HF^{+}(Y,\mathfrak{s}_{m-1,1})=\ta _{(d(m-1,1))}\oplus \ta _{(d(m-1,1)-1)}\oplus \FF _{(d(m-1,1)-1)}\\
& HF^{+}(Y,\mathfrak{s}_{i,j})=\ta _{(d(i,j))}\oplus \ta _{(d(i,j)-1)}\\
& HF^{+}(Y,\mathfrak{s}_{0,k})=\ta _{(d_{1}(m,k+1))}\oplus \ta _{(d_{1}(m,k+1)-1)}\\
& HF^{+}(Y,\mathfrak{s}_{l,0})=\ta _{(d_{1}(n,l+1))}\oplus \ta _{(d_{1}(n,l+1)-1)}
\end{eqnarray*}
for $1\leq i\leq m-1$, $1\leq j\leq n-1$, $0\leq k\leq n-2$, $0\leq l\leq m-2$ and $(i,j)\notin \{(m-1,1),(1,n-1)\}$, where the subscripts denote the absolute gradings of the bottom elements and 
\begin{xalignat*}{1}
& d(i,j)=\frac{m^{2}n+mn^{2}-4mn(i+j+1)+4n(i^{2}+2i)+4m(j^{2}+2j)-16ij}{4(mn-4)}\;,\\
& d_{1}(t,i)=\frac{m^{2}n+mn^{2}-4mni+4ti^{2}-4t}{4(mn-4)}\;.
\end{xalignat*}
The action of the exterior algebra $\Lambda ^{*}(H_{1}(Y,\ZZ )/\operatorname{Tors})$ on $HF^{+}(Y,\mathfrak{s})$ maps the first copy of $\ta $ isomorphically to the second copy in each torsion $\sp $ structure $\mathfrak{s}$, dropping the absolute grading of the generator by one. 
\end{theorem}
We use this result to determine whether the double plumbing $\n $ can occur inside some 4--manifolds $X$ with $H_{2}^{+}(X)=2$. If it can, the complement $W=X\backslash \operatorname{Int}(\n )$ is a negative semi-definite 4--manifold and \cite[Theorem 9.15]{OS4} gives an ob\-struction depending on the correction terms of $\y =\partial W$. 

In the manifold $X=\CC P^{2}\# \CC P^{2}$, every homology class $(x_{1},x_{2})\in H_{2}(X)$ with $(x_{1},x_{2})\in \{0,\pm 1,\pm 2\}^{2}\backslash \{(0,0)\}$ has a smooth representative of genus 0. Choosing two such representatives with algebraic intersection number 2, we check if they can have only 2 geometric intersections. 

Next we consider the manifold $X=S^{2}\times S^{2}\# S^{2}\times S^{2}$. According to Wall \cite{WALL}, every primitive homology class $(x_{1},x_{2},x_{3},x_{4})\in H_{2}(X)$ can be represented by an embedded sphere. We choose two such representatives with algebraic intersection number 2 and determine when the number of their geometric intersections has to be strictly greater than 2, thus not allowing the chosen homology classes to be represented by a double plumbing. We obtain the following estimates.
\begin{theorem}
\label{app}
\quad \\
a) Any two spheres representing classes $(2,2),(2,-1)\in H_{2}(\CC P^{2}\# \CC P^{2})$ intersect with at least $4$ geometric intersections, and there exist representatives with exactly $4$ intersections. \\
b) Let $t\in \mathbb{N}\backslash \{1\}$ and let $a$ be an odd positive integer. Any two spheres re\-pre\-sen\-ting classes $(a,2,0,0),(1,0,t,1)\in H_{2}(S^{2}\times S^{2}\# S^{2}\times S^{2})$ intersect with at least $4$  geometric intersections for all $a\geq 5$. 
\end{theorem}
By a similar method we study single plumbings of disk bundles over spheres inside a closed 4--manifold. An obstruction to embedding such configurations is based on the $d$--invariants of lens spaces. 
\begin{theorem}
\label{app1}
Let $k$ be a positive integer. Any two spheres representing classes $(2k+1,2,0,0),(-k,1,2k,1)\in H_{2}(S^{2}\times S^{2}\# S^{2}\times S^{2})$ intersect with at least 3 geometric intersections for all $k>1$.  
\end{theorem}
This paper is organized as follows. In Subsection \ref{HD} we describe a Heegaard diagram for $\y =\partial \n $. In \ref{CF} we present the corresponding chain complex $\widehat{CF}(\y )$ along with its decomposition into equivalence classes of $\sp $ structures and calculate the homology $HF^{+}(\y ,\mathfrak{s})$ in all torsion $\sp $ structures on $\y $. In Subsection \ref{absolute} we compute the absolute gradings $\gr $ of the generators of these groups which in turn determine the correction term  invariants $d_{b}(\y ,\mathfrak{s})$ for all torsion $\sp $ structures $\mathfrak{s}$ on $\y $. The first part of Section \ref{App} describes the general homological setting in which the double plumbing $\n $ arises as a submanifold in a closed 4--manifold $X$. In Subsection \ref{CP} we consider the case $X=\CC P^{2}\# \CC P^{2}$ and in Subsection \ref{S2} the case $X=S^{2}\times S^{2}\# S^{2}\times S^{2}$. In Section \ref{One} we investigate single plumbings of two disk bundles over spheres. We consider such configurations inside the manifold $\CC P^{2}\# \CC P^{2}$ in Subsection \ref{CP1} and inside the manifold $S^{2}\times S^{2}\# S^{2}\times S^{2}$ in Subsection \ref{S21}.
\newline

{\bf Acknowledgments:} I would like to thank my advisor Sa\v so Strle for all his help and support during our numerous discussions. I am also very grateful to the referees for a careful reading, many helpful comments and suggestions. 
\end{section}
\newpage

\begin{section}{Heegaard--Floer homology of the boundary of a double plumbing}
Let $N_{m,n}$ be the double plumbing of two disk bundles over spheres, where $m$ and $n$ denote the Euler numbers of the disk bundles contained in the plumbing. The base spheres of the bundles intersect twice inside the plumbing and we assume both intersections carry the same sign. Denote by $Y_{m,n}$ the boundary of $N_{m,n}$. Throughout this paper we assume $m,n\geq 4$. In this section we calculate $HF^{+}(Y_{m,n})$ with $\FF $ coefficients and prove Theorem \ref{th1}.  

\begin{subsection}{Heegaard diagram}
\label{HD}
Considering a Kirby diagram of the plumbing $N_{m,n}$ as a surgery diagram for $Y_{m,n}$, we derive the Heegaard diagram of its boundary. A disk bundle over a sphere is given by a single framed circle, and a double plumbing of two such bundles is represented by the Kirby diagram in Figure \ref{fig:kirby}. The second plumbing contributes a 1-handle. Instead of adding the 1-handle one can remove its complementary 2-handle with framing zero. The boundary of the resulting manifold remains unchanged if we replace the 1-handle by its complementary 2-handle with framing zero and obtain a Kirby diagram which is a link of three framed unknots $K_{1},K_{2}$ and $K_{3}$ in $S^{3}$. 
\begin{figure}[here]
\labellist
\normalsize \hair 2pt
\pinlabel $K_{1}$ at 120 440
\pinlabel $K_{2}$ at 700 440
\pinlabel $K_{3}$ at 400 320
\pinlabel $S^{2}$ at 800 190
\pinlabel $m$ [b] at 310 460
\pinlabel $n$ [b] at 530 460
\endlabellist
\begin{center}
\includegraphics[scale=0.50]{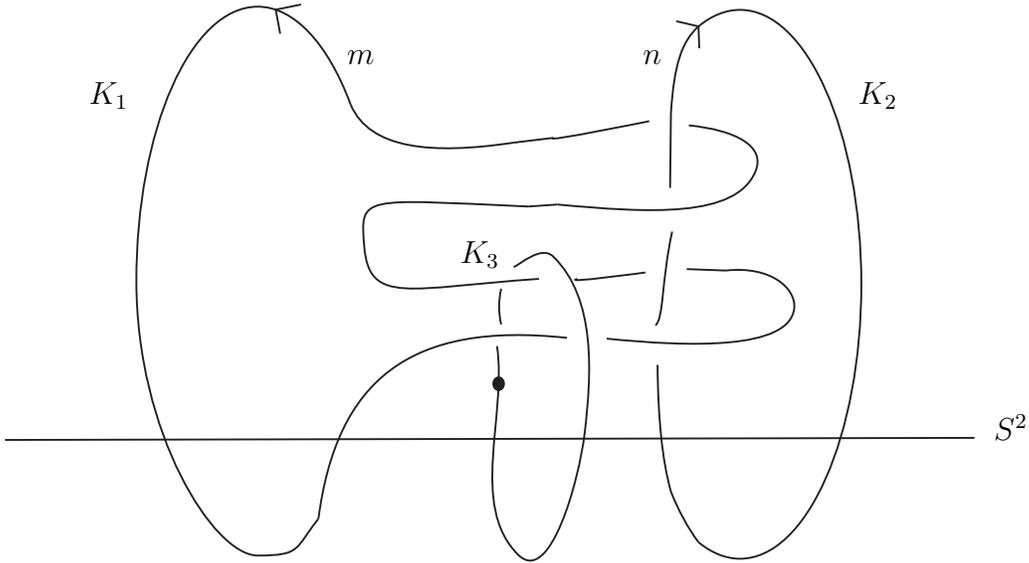}
\caption{The Kirby diagram of a double plumbing}
\label{fig:kirby}
\end{center}
\end{figure}
To obtain the Heegaard diagram of the boundary, we split the 3-sphere into two balls along the sphere $S^{2}$ shown in Figure \ref{fig:kirby}. Surgery along the three framed circles $K_{i}$ gives us two handlebodies of genus 3. The Heegaard diagram is drawn on the plane with three 1-handles added. The lower handlebody is a boundary connected sum of regular neighborhoods of the three circles $K_{1}$, $K_{2}$ and $K_{3}$. We denote by $\mu _{i}$ and $\lambda _{i}$ the meridian and longitude of the regular neighborhood of $K_{i}$ respectively. Each of the curves $\alpha _{i}$ is homologous to $\lambda _{i}$, and the curve $\beta _{i}$ corresponds to the framing of $K_{i}$ for $i=1,2,3$ (see Figure \ref{fig:heeg}). 

Thus, the first homology group $H_{1}(\y )$ is given by $$H_{1}(\y )=\left \langle \mu _{1},\mu _{2},\mu _{3}|\, m\mu _{1}+2\mu _{2}=0,\,2\mu _{1}+n\mu _{2}=0\right \rangle =\ZZ \langle \mu _{3}\rangle \oplus T\langle \mu _{1},\mu _{2}\rangle \;.$$ If at least one of the numbers $m,n$ is odd, the torsion group $T$ is cyclic and we get $T\langle \mu _{1},\mu _{2}\rangle =\ZZ _{mn-4}\langle \mu _{i}\rangle $ (if $m$ is odd and $n$ is even then $i=1$, if $n$ is odd and $m$ is even then $i=2$, if both $m$ and $n$ are odd then $i$ could either be $1$ or $2$). If both $m$ and $n$ are even numbers, then $T\langle \mu _{1},\mu _{2}\rangle =\ZZ _{\frac{mn-4}{2}}\langle \mu _{1}\rangle \oplus \ZZ _{2}\langle \frac{m}{2}\mu _{1}+\mu _{2}\rangle $. 
\end{subsection}

\begin{subsection}{The chain complex}
\label{CF}
We denote the intersections between the $\alpha $ and $\beta $ curves as follows (see Figure \ref{fig:heeg}): \\$\alpha _{1}\cap \beta _{1}=\{x_{1},x_{2},\ldots ,x_{m}\}$, $\alpha _{1}\cap \beta _{2}=\{y_{1},y_{2}\}$, $\alpha _{1}\cap \beta _{3}=\{u_{1},u_{2}\}$, $\alpha _{2}\cap \beta _{1}=\{a_{1},a_{2}\}$, $\alpha _{2}\cap \beta _{2}=\{b_{1},b_{2},\ldots ,b_{n}\}$, $\alpha _{2}\cap \beta _{3}=\{c_{1},c_{2}\}$, $\alpha _{3}\cap \beta _{1}=\{d_{1},d_{2}\}$, $\alpha _{3}\cap \beta _{2}=\{e_{1},e_{2}\}$, $\alpha _{3}\cap \beta _{3}=\{f_{1},f_{2}\}$. \\
The chain complex $\widehat{CF}(Y)$ is generated by unordered triples $$\TT _{\alpha }\cap \TT _{\beta }=\{\{x_{i},b_{j},f_{k}\},\{x_{i},c_{k},e_{l}\},\{y_{k},a_{l},f_{r}\},\{y_{k},c_{l},d_{r}\},\{u_{k},a_{l},e_{r}\},\{u_{k},b_{j},d_{l}\}\}$$ for $k,l,r\in \{1,2\}$ and $i=1,\ldots ,m$ and $j=1,\ldots n$. 

The complement of the $\alpha $ and $\beta $ curves in the Heegaard diagram is a disjoint union of elementary domains. There are two regions in the diagram where a curve $\beta _{i}$ winds around a hole in the direction of $\mu _{i}$; we denote the elementary domains in the winding region of $\beta _{1}$ by $A_{1},\ldots ,A_{m-3}$ and the elementary domains in the winding region of $\beta _{2}$ by $B_{1},\ldots ,B_{n-3}$. The remaining elementary domains of the Heegaard diagram are denoted by $D_{1},\ldots ,D_{16}$. They consist of five hexagons $D_{1},D_{4},D_{8},D_{9}$ and $D_{16}$, one dodecagon $D_{5}$ and one bigon $D_{13}$; all the remaining elementary domains are rectangles. We put the basepoint $z$ of the Heegaard diagram into the elementary domain $D_{5}$. There is a single periodic domain in our diagram, bounded by the difference $\alpha _{3}-\beta _{3}$ of the two homologous curves, which is given by the sum $$\mathcal {P}=D_{3}+D_{4}+D_{6}+D_{9}+D_{11}+D_{12}-D_{13}+D_{14}+D_{16}+B_{1}+B_{2}+\ldots +B_{n-3}\;.$$
Applying the first Chern class formula \cite[Proposition 7.5]{OS2} we obtain 
\begin{align*}
\left \langle c_{1}(\mathfrak{s}_{z}(\mathbf{x})),\mathcal {P}\right \rangle & =\chi (\mathcal{P})+2\sum _{x_{i}\in \mathbf{x}}\overline{n}_{x_{i}}(\mathcal{P})=3(1-\frac {6}{4})+(-1)(1-\frac {2}{4})+2\sum _{x_{i}\in \mathbf{x}}\overline{n}_{x_{i}}(\mathcal{P})=\\
&=2\left (\sum _{x_{i}\in \mathbf{x}}\overline{n}_{x_{i}}(\mathcal{P})-1\right )\;,
\end{align*}
\newpage
\newgeometry{margin=0cm}
\thispagestyle{empty} 
\begin{figure}
\begin{center}
\labellist
\scriptsize \hair 2pt
\pinlabel \rotatebox{90}{$\alpha _{1}$} [t] at 365 316
\pinlabel \rotatebox{90}{$\alpha _{2}$} [b] at 290 470
\pinlabel \rotatebox{90}{$\alpha _{3}$} [t] at 338 606
\pinlabel \rotatebox{90}{$\beta _{1}$} at 125 400
\pinlabel \rotatebox{90}{$\beta _{2}$} at 170 450
\pinlabel \rotatebox{90}{$\beta _{3}$} at 216 440
\pinlabel \rotatebox{90}{$x_{1}$} [b] at 283 316
\pinlabel \rotatebox{90}{$x_{2}$} [b] at 300 316
\pinlabel \rotatebox{90}{$x_{m-2}$} [b] at 318 316
\pinlabel \rotatebox{90}{$x_{m-1}$} [b] at 333 316
\pinlabel \rotatebox{90}{$x_{m}$} [b] at 372 316
\pinlabel \rotatebox{90}{$y_{1}$} [b] at 350 316
\pinlabel \rotatebox{90}{$y_{2}$} [b] at 395 316
\pinlabel \rotatebox{90}{$u_{1}$} [b] at 385 316
\pinlabel \rotatebox{90}{$u_{2}$} [t] at 400 314
\pinlabel \rotatebox{90}{$d_{1}$} [b] at 312 470
\pinlabel \rotatebox{90}{$d_{2}$} [b] at 342 470
\pinlabel \rotatebox{90}{$f_{1}$} [b] at 362 470
\pinlabel \rotatebox{90}{$f_{2}$} [b] at 375 470
\pinlabel \rotatebox{90}{$e_{1}$} [b] at 326 470
\pinlabel \rotatebox{90}{$e_{2}$} [t] at 358 468
\pinlabel \rotatebox{90}{$a_{1}$} [b] at 340 610
\pinlabel \rotatebox{90}{$a_{2}$} [b] at 382 610
\pinlabel \rotatebox{90}{$c_{1}$} [b] at 366 610
\pinlabel \rotatebox{90}{$c_{2}$} [b] at 402 610
\pinlabel \rotatebox{90}{$b_{1}$} [b] at 280 610
\pinlabel \rotatebox{90}{$b_{2}$}  [b] at 296 610
\pinlabel \rotatebox{90}{$b_{n-1}$} [b] at 320 610
\pinlabel \rotatebox{90}{$b_{n}$} [b] at 354 610
\pinlabel \rotatebox{90}{$\mathcal{D}_{1}$} at 470 470 
\pinlabel \rotatebox{90}{$\mathcal{D}_{2}$} at 420 260
\pinlabel \rotatebox{90}{$\mathcal{D}_{2}$} at 296 725
\pinlabel \rotatebox{90}{$\mathcal{D}_{3}$} at 387 250
\pinlabel \rotatebox{90}{$\mathcal{D}_{3}$} at 296 711
\pinlabel \rotatebox{90}{$\mathcal{D}_{4}$} at 250 500
\pinlabel \rotatebox{90}{$\mathcal{D}_{4}$} at 296 690
\pinlabel \rotatebox{90}{$\mathcal{D}_{4}$} at 373 252
\pinlabel \rotatebox{90}{$\mathcal{D}_{5}$} at 230 424
\pinlabel \rotatebox{90}{$\mathcal{D}_{5}$} at 363 260
\pinlabel \rotatebox{90}{$\mathcal{D}_{5}$} at 374 424
\pinlabel \rotatebox{90}{$\mathcal{D}_{6}$} at 200 612
\pinlabel \rotatebox{90}{$\mathcal{D}_{7}$} at 345 270
\pinlabel \rotatebox{90}{$\mathcal{D}_{8}$} at 320 390
\pinlabel \rotatebox{90}{$\mathcal{D}_{9}$} at 315 530  
\pinlabel \rotatebox{90}{$\mathcal{D}_{10}$} at 352 380
\pinlabel \rotatebox{90}{$\mathcal{D}_{11}$} at 352 560 
\pinlabel \rotatebox{90}{$\mathcal{D}_{12}$} at 368 552
\pinlabel \rotatebox{90}{$\mathcal{D}_{13}$} at 372 460 
\pinlabel \rotatebox{90}{$\mathcal{D}_{14}$} at 430 470
\pinlabel \rotatebox{90}{$\mathcal{D}_{15}$} at 420 520 
\pinlabel \rotatebox{90}{$\mathcal{D}_{16}$} at 441 520
\pinlabel \rotatebox{90}{$\mathcal{A}_{1}$} at 294 335 
\pinlabel \rotatebox{90}{$\mathcal{B}_{1}$} at 290 628 
\pinlabel \rotatebox{60}{$\mathcal{A}_{m-3}$} at 320 295 
\pinlabel \rotatebox{70}{$\mathcal{B}_{n-3}$} at 320 596 
\pinlabel \rotatebox{90}{$\mu _{1}$} at 253 318 
\pinlabel \rotatebox{90}{$\mu _{2}$} at 252 606 
\pinlabel \rotatebox{90}{$\mu _{3}$} at 256 472 
\pinlabel \rotatebox{90}{$z$} at 270 443 
\endlabellist
\centering
\caption{The Heegaard diagram of $\y $}
\includegraphics[scale=1.0]{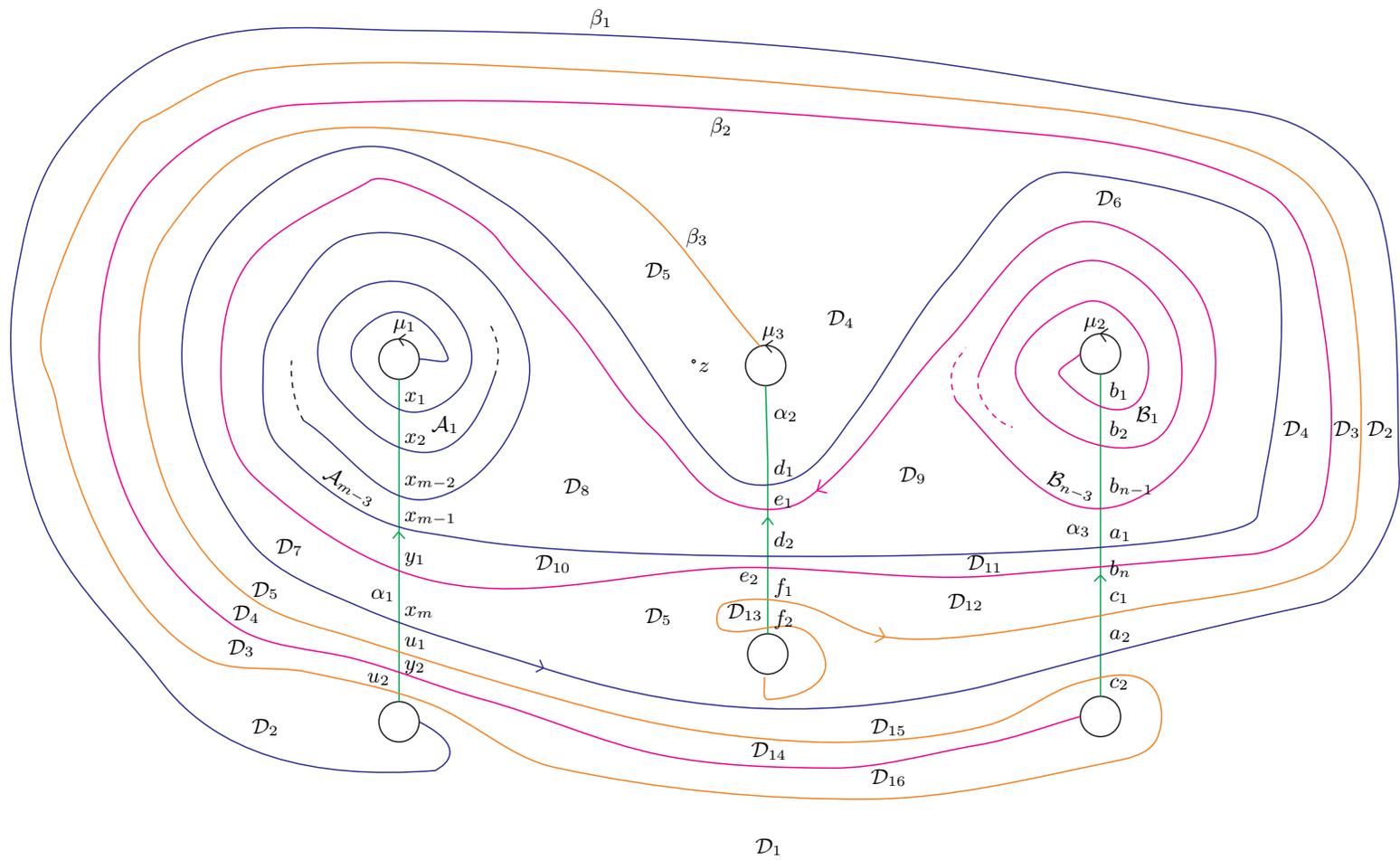}
\label{fig:heeg}
\end{center}
\end{figure}
\restoregeometry

\newpage
\thispagestyle{plain}
based on which we can determine the torsion $\sp $ structures. A generator $\mathbf{x}\in \TT _{\alpha }\cap \TT _{\beta }$ belongs to a torsion $\sp $ structure if and only if $\left \langle c_{1}(\mathfrak{s}_{z}(\mathbf{x})),\mathcal {P}\right \rangle =0$, which happens exactly when $\sum _{x_{i}\in \mathbf{x}}\overline{n}_{x_{i}}(\mathcal{P})=1$. Thus, the torsion $\sp $ structures of our chain complex contain the following generators:
$$\{\{x_{i},b_{j},f_{k}\},\{x_{i},c_{k},e_{r}\},\{y_{k},a_{k},f_{r}\},\{y_{1},c_{k},d_{r}\},\{u_{k},a_{2},e_{r}\}\}$$ for $k,r\in \{1,2\}$ and $i=1,\ldots ,m$ and $j=1,\ldots n$. There are $2(mn+2m+6)$ generators of the torsion $\sp $ structures on $\y $. 

\begin{subsubsection}{Notation for $\sp $ structures}
\label{notspin}
There is a one-to-one correspondence $$\delta ^{\tau }\colon \sp (\y )\to H^{2}(\y )$$ \cite[Subsection 2.6]{OS1}. Thus, we may identify the $\sp $ structures on $\y $ with co\-ho\-mo\-lo\-gy classes, or even with their Poincar\'{e} dual homology classes in $H_{1}(\y )$. Then the natural map $c_{1}\colon \sp (\y )\to H^{2}(\y )$ assigning to any $\sp $ structure its first Chern class is connected to $\delta ^{\tau }$ by $c_{1}(\mathfrak{s})=2\delta ^{\tau }(\mathfrak{s})$.  

Similarly, a $\sp $ structure on a 4--manifold $W$ has a determinant line bundle whose first Chern class is a characteristic element in $H^{2}(W)$. For every characteristic element $c\in H^{2}(W)$ there exists a $\sp $ structure on $W$ with determinant line bundle whose first Chern class is equal to $c$ \cite[Proposition 2.4.16]{kirby}. If $H^{2}(W)$ contains no 2-torsion, then such a $\sp $ structure is unique. If $W$ is a 4--manifold with boundary $\y $, we will identify the restriction $\sp (W)\to \sp (\y )$ with the corresponding restriction map on cohomology $H^{2}(W)\to H^{2}(\y )$ after making an appropriate choice of origins in the sets of $\sp $ structures. 

Let $s_{1},s_{2}\in H_{2}(\n )$ be the homology classes of the base spheres in the double plumbing. As we will see in Section \ref{App}, all the torsion $\sp $ structures on $\y $ extend to $\sp $ structures on $\n $. For two integers $i$ and $j$, denote by $\mathfrak{t}_{i,j}$ the unique $\sp $ structure on $\n $ for which 
\begin{align*}
\langle c_{1}(\mathfrak{t}_{i,j}),s_{1}\rangle +m=2i\\
\langle c_{1}(\mathfrak{t}_{i,j}),s_{2}\rangle +n=2j
\end{align*}   
and let $\mathfrak{s}_{i,j}=\mathfrak{t}_{i,j}|_{\y }$. 

\begin{corollary} 
\label{unique}
All the torsion $\sp $ structures on $\y $ are uniquely determined by
\begin{xalignat*}{1}
& \mathfrak{s}_{i,j} \textrm{ for } 1\leq i\leq m-1,\, 1\leq j\leq n-1\\
& \mathfrak{s}_{0,j} \textrm{ for } 0\leq j\leq n-2 \\
& \mathfrak{s}_{i,0} \textrm{ for } 0\leq i\leq m-2
\end{xalignat*} with identifications $\mathfrak{s}_{m-2,0}=\mathfrak{s}_{0,n-2}$ and $\mathfrak{s}_{m-1,1}=\mathfrak{s}_{1,n-1}$.
\end{corollary}
\begin{proof}
Denote by $s_{1}^{*},s_{2}^{*}$ the basis for $H^{2}(\n )$ which is Hom-dual to the basis $s_{1},s_{2}$. Then the restriction map $H^{2}(\n )\to H^{2}(\y )$ maps $s_{i}^{*}\mapsto PD(\mu _{i})$ for $i=1,2$. Remember the presentation of the first homology group $$H_{1}(\y )=\left \langle \mu _{1},\mu _{2},\mu _{3}|\, m\mu _{1}+2\mu _{2}=0,\,2\mu _{1}+n\mu _{2}=0\right \rangle \;.$$ Its torsion subgroup is uniquely determined by the classes $i\mu _{1}+j\mu _{2}$ for $1\leq i\leq m-1$ and $1\leq j\leq n-1$, $i\mu _{1}$ for $0\leq i\leq m-2$ and $j\mu _{2}$ for $0\leq j\leq n-2$ with identifications $(m-2)\mu _{1}=(n-2)\mu _{2}$ and $(m-1)\mu _{1}+\mu _{2}=\mu _{1}+(n-1)\mu _{2}$. 
\end{proof}

The set $\TT _{\alpha }\cap \TT _{\beta }$ is decomposed into equivalence classes according to the $\epsilon $-relation \cite[Definition 2.11]{OS1}. Furthermore, by the map $\mathfrak{s}_{z}\colon \TT _{\alpha }\cap \TT _{\beta }\to H^{2}(\y )$, every equivalence class of generators together with a fixed basepoint $z$ determines a $\sp $ structure and its corresponding cohomology class. By \cite[Lemma 2.19]{OS1} for any two generators $\mathbf{x},\mathbf{y}\in \TT _{\alpha }\cap \TT _{\beta }$ the following holds: $$\mathfrak{s}_{z}(\mathbf{y})-\mathfrak{s}_{z}(\mathbf{x})=PD[\epsilon (\mathbf{x},\mathbf{y})]\;.$$ 
\end{subsubsection}

From the Heegaard diagram we obtain

\begin{align*}
& \epsilon (\{x_{i+1},-\},\{x_{i},-\})=\mu _{1} \textrm { for $1\leq i\leq m-2$}\\
& \epsilon (\{x_{m},-\},\{x_{m-1},-\})=\epsilon (\{b_{n},-\},\{b_{n-1},-\})=\mu _{1}+\mu _{2}\\
& \epsilon (\{x_{m},-\},\{x_{1},-\})=(m-1)\mu _{1}+\mu _{2}=-\mu _{1}-\mu _{2}\\
& \epsilon (\{b_{i+1},-\},\{b_{i},-\})=\mu _{2}\textrm { for $1\leq i\leq n-2$}\\
& \epsilon (\{b_{n},-\},\{b_{1},-\})=(n-1)\mu _{2}+\mu _{1}=-\mu _{1}-\mu _{2}\\
& \epsilon (\{a_{2},-\},\{a_{1},-\})=\mu _{1}+\mu _{2}-\mu _{3}\\
& \epsilon (\{y_{2},-\},\{y_{1},-\})=\mu _{1}+\mu _{2}+\mu _{3}\\
& \epsilon (\{c_{2},-\},\{c_{1},-\})=\epsilon (\{e_{2},-\},\{e_{1},-\})=\mu _{1}\\
& \epsilon (\{d_{1},-\},\{d_{2},-\})=\epsilon (\{u_{2},-\},\{u_{1},-\})=\mu _{2}
\end{align*}

Suppose that $m,n\geq 4$. The equivalence classes of generators in the torsion $\sp $ structures, given by the $\epsilon $-relations above, are given below. For now, we will choose the origin for the $\sp $ structures arbitrarily and call it $\mathfrak{s}_{0}$. We will show later that $\mathfrak{s}_{0}$ is in fact the $\sp $ structure $\mathfrak{s}_{1,n-1}=\mathfrak{s}_{m-1,1}$ (see Lemma \ref{lemmasp}). 
\begin{xalignat*}{1}
& \mathfrak{s}_{0}+\mu _{1}+\mu _{2}\colon \, \{x_{m-1},c_{1},e_{2}\}\sim \{x_{m-1},c_{2},e_{1}\}\sim \{x_{m-2},c_{2},e_{2}\}\sim \{x_{m-1},b_{n},f_{1}\}\sim \\
& \sim \{x_{m-1},b_{n},f_{2}\}\sim \{x_{m},b_{n-1},f_{1}\}\sim \{x_{m},b_{n-1},f_{2}\}\sim \{y_{1},c_{1},d_{2}\}\sim \{y_{1},a_{1},f_{1}\}\sim \\
& \sim \{y_{1},a_{1},f_{2}\}\\
& \quad \\
& \mathfrak{s}_{0}+\mu _{1}\colon \, \{x_{m-2},b_{1},f_{1}\}\sim \{x_{m-2},b_{1},f_{2}\}\sim \{y_{1},c_{1},d_{1}\}\sim \{x_{m},c_{1},e_{1}\}\\
& \quad \\
& \mathfrak{s}_{0}+\mu _{2}\colon \, \{x_{m-1},c_{2},e_{2}\}\sim \{x_{1},b_{n-2},f_{1}\}\sim \{x_{1},b_{n-2},f_{2}\}\sim \{y_{1},c_{2},d_{2}\}\\
& \quad \\
&\mathfrak{s}_{0}\colon \, \{x_{m},c_{1},e_{2}\}\sim \{x_{m},c_{2},e_{1}\}\sim \{x_{m},b_{n},f_{1}\}\sim \{x_{m},b_{n},f_{2}\}\sim \{x_{m-1},b_{1},f_{1}\}\sim \\
& \sim \{x_{m-1},b_{1},f_{2}\}\sim \{x_{1},b_{n-1},f_{1}\}\sim \{x_{1},b_{n-1},f_{2}\}\sim \{y_{1},c_{2},d_{1}\}\sim \{u_{1},a_{2},e_{1}\}\\
& \quad \\
& \mathfrak{s}_{0}-\mu _{2}\colon \, \{x_{1},c_{1},e_{1}\}\sim \{x_{m-1},b_{2},f_{1}\}\sim \{x_{m-1},b_{2},f_{2}\}\sim \{u_{2},a_{2},e_{1}\}\\
& \quad \\
& \mathfrak{s}_{0}-\mu _{1}\colon \, \{x_{m},c_{2},e_{2}\}\sim \{x_{2},b_{n-1},f_{1}\}\sim \{x_{2},b_{n-1},f_{2}\}\sim \{u_{1},a_{2},e_{2}\}\\
& \quad \\
& \mathfrak{s}_{0}-\mu _{1}-\mu _{2}\colon \, \{x_{1},c_{1},e_{2}\}\sim \{x_{2},c_{1},e_{1}\}\sim \{x_{1},c_{2},e_{1}\}\sim \{x_{1},b_{n},f_{1}\}\sim \{x_{1},b_{n},f_{2}\}\\
& \sim \{x_{m},b_{1},f_{1}\}\sim \{x_{m},b_{1},f_{2}\}\sim \{y_{2},a_{2},f_{1}\}\sim \{y_{2},a_{2},f_{2}\}\sim \{u_{2},a_{2},e_{2}\}\\
& \quad \\
& \mathfrak{s}_{0}-i\mu _{1}-\mu _{2}\colon \, \{x_{i},b_{n},f_{1}\}\sim \{x_{i},b_{n},f_{2}\}\sim \{x_{i},c_{1},e_{2}\}\sim \{x_{i},c_{2},e_{1}\}\sim \\ & \sim \{x_{i+1},c_{1},e_{1}\}\sim \{x_{i-1},c_{2},e_{2}\}\quad \textrm{ for $2\leq i\leq m-2$}\\
& \quad  \\
& \textrm {other classes}\colon \, \{x_{r},b_{j},f_{1}\}\sim \{x_{r},b_{j},f_{2}\}
\end{xalignat*} 
for $(r,j)\in \{1,\ldots ,m\}\times \{1,\ldots ,n-1\}\backslash \\ \{(m,n-1),(m-2,1),(1,n-2),(m-1,1),(1,n-1),(m-1,2),(2,n-1),(m,1)\}$. 

We require $m,n\geq 4$ so that the $\sp $ equivalence classes listed above are all distinct. This is still true if $m\geq 4$ and $n=3$ or vice versa. By this requirement we also make sure that the intersection form of the double plumbing $\n $ is positive definite, which will be important in Section \ref{App}. 

Now we consider the differentials of the chain complex $CF^{+}(\y )$. For every torsion $\sp $ structure $\mathfrak{s}\in \sp (\y )$, we draw a schematic depicting the generators of $\widehat{CF}(\y ,\mathfrak{s})$ vertically according to their relative Maslov grading. Then we list the nonnegative domains of Whitney disks with Maslov index 1 between the generators. If between two generators there is only one such domain of a disk which has a unique holomorphic representative, we denote it by an arrow in the schematic. When neccessary, we also list some domains of Whitney disks with negative coefficients. We denote the action of $\Lambda ^{*}(H_{1}(\y ,\ZZ )/\operatorname{Tors})$ by a dotted arrow in each schematic. Using this information we calculate the homology $HF^{+}(\y ,\mathfrak{s})$. 

When the first Betti number of a 3--manifold $Y$ is at most 2, the homology $HF^{\infty }(Y,\mathfrak{s})$ in a torsion $\sp $ structure $\mathfrak{s}$ is determined by the integral homology of $Y$ \cite[Theorem 10.1]{OS2}. Since $H^{1}(\y ;\ZZ )\cong \ZZ $, it follows that for every torsion $\sp $ structure $\mathfrak{s}$, the homology $HF^{\infty }$ of our manifold is given by $$HF^{\infty }(\y ,\mathfrak{s})\cong \FF [U,U^{-1}]\otimes _{\ZZ }\Lambda ^{*}\FF $$ and has two $\FF [U,U^{-1}]$ summands.

\noindent \textbf{Classes with two generators}
\begin{displaymath}
\xymatrix{ \{x_{r},b_{j},f_{2}\} \ar@{.>}[d]\\
\{x_{r},b_{j},f_{1}\}}
\end{displaymath}
\begin{xalignat*}{1}
& \{-,-,f_{2}\}\to \{-,-,f_{1}\}\colon D_{13}\textrm{ and }D_{13}+\mathcal{P}\\
& \{-,-,f_{1}\}\to \{-,-,f_{2}\}\colon \phi =D_{1}+D_{2}+D_{5}+D_{7}+D_{8}+D_{10}+D_{13}+D_{15}+\\
& A_ {1}+A_{2}+\ldots + A_{m-3} \textrm{ and }\phi +\mathcal{P}
\end{xalignat*} 
As observed above, $HF^{\infty }(\y ,\mathfrak{s})$ has two $\FF [U,U^{-1}]$ summands, so the differential $\partial ^{\infty }$ in this class has to be trivial. This means there is an even number of holomorphic disks from $[\{x_{r},b_{j},f_{2}\},i]$ to $[\{x_{r},b_{j},f_{1}\},i]$. There is a disk from $\{x_{r},b_{j},f_{2}\}$ to $\{x_{r},b_{j},f_{1}\}$ with a bigonal domain $D_{13}$. By the Riemann mapping theorem, this disk has a unique holomorphic representative.  The domain of any disk $\phi $ from $\{x_{r},b_{j},f_{2}\}$ to $\{x_{r},b_{j},f_{1}\}$ is given by $\mathcal{D}(\phi )=D_{13}+a\mathcal{P}+b\Sigma $ for two integers $a$ and $b$. Such a disk has Maslov index $\mu (\phi )=1+2b$ and $n_{z}(\phi )=b$. It follows that the Maslov index equals 1 if and only if $n_{z}(\phi )=b=0$, which implies that the domain $D_{13}+a\mathcal{P}$ has only non-negative multiplicities when $a\in \{0,1\}$. Thus the domain of the se\-cond holomorphic disk from $\{x_{r},b_{j},f_{2}\}$ to $\{x_{r},b_{j},f_{1}\}$ is $D_{13}+\mathcal{P}$. This disk has an odd number of holomorphic representatives. The differential of the chain complex $CF^{+}(\y )$ in this class is trivial: $\partial ^{+}[\{x_{r},b_{j},f_{2}\},i]=\partial ^{+}[\{x_{r},b_{j},f_{1}\},i]=0$ and it follows that $$HF^{+}(\y ,\mathfrak{s})=\ta \oplus \ta $$ is freely generated by the elements $[\{x_{r},b_{j},f_{2}\},i]$ and $[\{x_{r},b_{j},f_{1}\},i]$ for $i\geq 0$. 

\noindent \textbf{Class $\mathfrak{s}_{0}+\mu _{1}$}
\begin{displaymath}
\xymatrix{ \{x_{m},c_{1},e_{1}\} \ar@{->}[d] & \quad \\
\{y_{1},c_{1},d_{1}\} & \{x_{m-2},b_{1},f_{2}\} \ar@{.>}[d]\\
\quad & \{x_{m-2},b_{1},f_{1}\}}
\end{displaymath}
\begin{xalignat*}{1}
& \{x_{m},c_{1},e_{1}\}\to \{x_{m-2},b_{1},f_{2}\}\colon D_{1}+D_{2}+D_{3}+D_{6}+D_{7}+D_{9}-D_{10}+D_{12}-\\
& -D_{13}+D_{16}+A_{1}+\ldots +A_{m-3}+B_{1}+\ldots +B_{n-3}\,,\textrm{ no holom. representatives}\\
& \{x_{m},c_{1},e_{1}\}\to \{y_{1},c_{1},d_{1}\}\colon D_{7}\, ,\textrm{ rectangle}\\
& \{y_{1},c_{1},d_{1}\} \to \{x_{m-2},b_{1},f_{1}\}\colon D_{1}+D_{2}+D_{3}+D_{6}+D_{9}-D_{10}+D_{12}+D_{16}+\\
& +A_{1}+\ldots +A_{m-3}+B_{1}+\ldots +B_{n-3}\, ,\textrm{ no holomorphic representatives}\\
& \{y_{1},c_{1},d_{1}\}\to \{x_{m},c_{1},e_{1}\}\colon D_{1}+D_{2}+D_{5}+D_{8}+D_{10}+2D_{13}+D_{15}+A_{1}+\\
& +\ldots +A_{m-3}\;, \phi +\mathcal{P}\textrm{ and }\phi +2\mathcal{P}\\
& \{-,-,f_{2}\}\to \{-,-,f_{1}\}\colon D_{13}\textrm{ and }D_{13}+\mathcal{P}\\
& \{x_{m-2},b_{1},f_{2}\}\to \{x_{m},c_{1},e_{1}\}\colon \phi =D_{4}+D_{5}+D_{8}+2D_{10}+D_{11}+2D_{13}+D_{14}+\\
& +D_{15}, \phi +\mathcal{P}\textrm{ and }\phi +2\mathcal{P}\\
& \{-,-,f_{1}\}\to \{-,-,f_{2}\}\colon \phi =D_{1}+D_{2}+D_{5}+D_{7}+D_{8}+D_{10}+D_{13}+D_{15}+\\
& +A_{1}+\ldots +A_{m-3}\textrm{ and }\phi +\mathcal{P}\\
& \{x_{m-2},b_{1},f_{1}\}\to \{y_{1},c_{1},d_{1}\}\colon \phi =D_{4}+D_{5}+D_{7}+D_{8}+2D_{10}+D_{11}+D_{13}+\\
& +D_{14}+D_{15}\textrm{ and }\phi +\mathcal{P}
\end{xalignat*}
There is a disk from $\{x_{m},c_{1},e_{1},\}$ to $\{y_{1},c_{1},d_{1}\}$ with a rectangular domain and a unique holomorphic representative. There is no disk from $\{x_{m},c_{1},e_{1}\}$ to $\{x_{m-2},b_{1},f_{2}\}$ whose domain would be non-negative, so these disks have no holomorphic representatives. Thus $\partial ^{+}[\{x_{m},c_{1},e_{1}\},i]=[\{y_{1},c_{1},d_{1}\},i]$ and  $\partial ^{+}[\{y_{1},c_{1},d_{1}\},i]=0$. Similarly we have $\partial ^{\infty }[\{x_{m},c_{1},e_{1}\},i]=[\{y_{1},c_{1},d_{1}\},i]$ and $\partial ^{\infty }[\{y_{1},c_{1},d_{1}\},i]=0$, so $HF^{\infty }(\y ,\mathfrak{s}_{0}+\mu _{1})$ is generated by $[\{x_{m-2},b_{1},f_{2}\},i]$ and $[\{x_{m-2},b_{1},f_{1}\},i]$. We already know there is an even number of holomorphic disks from $\{x_{m-2},b_{1},f_{2}\}$ to $\{x_{m-2},b_{1},f_{1}\}$, so $\partial ^{+}[\{x_{m-2},b_{1},f_{2}\},i]=\partial ^{+}[\{x_{m-2},b_{1},f_{1}\},i]=0$. It follows that $$HF^{+}(Y,\mathfrak{s}_{0}+\mu _{1})\cong \ta \oplus \ta $$ is freely generated by the elements $[\{x_{m-2},b_{1},f_{2}\},i]$ and $[\{x_{m-2},b_{1},f_{1}\},i]$ for $i\geq 0$. 

\noindent \textbf{Class $\mathfrak{s}_{0}+\mu _{2}$}
\begin{displaymath}
\xymatrix {\{x_{m-1},c_{2},e_{2}\} \ar@{->}[d] & \quad \\
\{y_{1},c_{2},d_{2}\} & \{x_{1},b_{n-2},f_{2}\} \ar@{.>}[d]\\
\quad & \{x_{1},b_{n-2},f_{1}\}}
\end{displaymath}
\begin{xalignat*}{1}
&  \{x_{m-1},c_{2},e_{2}\}\to \{x_{1},b_{n-2},f_{2}\}\colon -D_{2}-D_{6}-D_{7}+D_{10}+D_{11}+D_{12}-D_{13}+\\
& +D_{14}+D_{15}+D_{16}+B_{1}+\ldots +B_{n-3}\, ,\textrm{ no holomorphic representatives}\\
& \{x_{m-1},c_{2},e_{2}\}\to \{y_{1},c_{2},d_{2}\}\colon D_{10}\, ,\textrm{ rectangle}\\
& \{y_{1},c_{2},d_{2}\}\to \{x_{1},b_{n-2},f_{1}\}\colon -D_{2}-D_{6}-D_{7}+D_{11}+D_{12}+D_{14}+D_{15}+\\
& +D_{16}+B_{1}+\ldots +B_{n-3}\, ,\textrm{ no holomorphic representatives}\\
& \{y_{1},c_{2},d_{2}\}\to \{x_{m-1},c_{2},e_{2}\}\colon \phi =D_{1}+D_{2}+D_{5}+D_{7}+D_{8}+2D_{13}+D_{15}+\\
& +A_{1}+\ldots +A_{m-3}, \phi +\mathcal{P} \textrm{ and }\phi +2\mathcal{P}\\
& \{x_{1},b_{n-2},f_{2}\}\to \{x_{m-1},c_{2},e_{2}\}\colon \phi =D_{1}+2D_{2}+D_{3}+D_{4}+D_{5}+2D_{6}+2D_{7}+\\
& +D_{8}+D_{9}+2D_{13}+A_{1}+\ldots +A_{m-3}, \phi +\mathcal{P} \textrm{ and }\phi +2\mathcal{P}\\
& \{-,-,f_{2}\}\to \{-,-,f_{1}\}\colon D_{13}\textrm{ and }D_{13}+\mathcal{P}\\
& \{x_{1},b_{n-2},f_{1}\}\to \{y_{1},c_{2},d_{2}\}\colon \phi =D_{1}+2D_{2}+D_{3}+D_{4}+D_{5}+2D_{6}+2D_{7}+D_{8}+\\
& +D_{9}+D_{10}+D_{13}+A_{1}+\ldots +A_{m-3}\textrm{ and }\phi +\mathcal{P}\\
& \{x_{1},b_{n-2},f_{1}\}\to \{x_{1},b_{n-2},f_{2}\}\colon \phi =D_{1}+D_{2}+D_{5}+D_{7}+D_{8}+D_{10}+D_{13}+\\
& +D_{15}+A_{1}+\ldots +A_{m-3}\textrm{ and }\phi +\mathcal{P}
\end{xalignat*}
Using analogous reasoning as above gives the homology $HF^{+}(Y,\mathfrak{s}_{0}+\mu _{2})\cong \ta \oplus \ta $, freely generated by the elements $[\{x_{1},b_{n-2},f_{2}\},i]$ and $[\{x_{1},b_{n-2},f_{1}\},i]$ for $i\geq 0$. 

\noindent \textbf{Class $\mathfrak{s}_{0}-\mu _{2}$}
\begin{displaymath}
\xymatrix{ \{x_{1},c_{1},e_{1}\} \ar@{->}[d]& \quad \\
\{u_{2},a_{2},e_{1}\} & \{x_{m-1},b_{2},f_{2}\} \ar@{.>}[d]\\
\quad & \{x_{m-1},b_{2},f_{1}\}}
\end{displaymath}
\begin{xalignat*}{1}
& \{x_{1},c_{1},e_{1}\}\to \{u_{2},a_{2},e_{1}\}\colon D_{2}\, ,\textrm{ rectangle}\\
& \{x_{1},c_{1},e_{1}\}\to \{x_{m-1},b_{2},f_{2}\}\colon D_{2}+D_{3}+D_{6}+D_{7}+D_{9}-D_{10}+D_{12}-\\
& -D_{13}-D_{14}-D_{15}+B_{1}+\ldots +B_{n-3}\, ,\textrm{ no holomorphic representatives}\\
& \{u_{2},a_{2},e_{1}\}\to \{x_{1},c_{1},e_{1}\}\colon \phi =D_{1}+D_{5}+D_{7}+D_{8}+D_{10}+2D_{13}+D_{15}+\\
& +A_{1}+\ldots +A_{m-3}\;, \phi +\mathcal{P}\textrm{ and }\phi +2\mathcal{P}\\
& \{u_{2},a_{2},e_{1}\}\to \{x_{m-1},b_{2},f_{1}\}\colon D_{3}+D_{6}+D_{7}+D_{9}-D_{10}+D_{12}-D_{14}-\\
& -D_{15}+B_{1}+\ldots +B_{n-3}\, ,\textrm{ no holomorphic representatives}\\
& \{x_{m-1},b_{2},f_{2}\}\to \{x_{1},c_{1},e_{1}\}\colon \phi =D_{1}+D_{4}+D_{5}+D_{8}+2D_{10}+D_{11}+2D_{13}+\\
& +2D_{14}+2D_{15}+D_{16}+A_{1}+\ldots +A_{m-3}\;, \phi +\mathcal{P}\textrm{ and }\phi +2\mathcal{P}\\
& \{x_{m-1},b_{2},f_{2}\}\to \{x_{m-1},b_{2},f_{1}\}\colon D_{13}\textrm{ and }D_{13}+\mathcal{P}\\
& \{x_{m-1},b_{2},f_{1}\}\to \{u_{2},a_{2},e_{1}\}\colon \phi =D_{1}+D_{2}+D_{4}+D_{5}+D_{8}+2D_{10}+D_{11}+\\
& +D_{13}+2D_{14}+2D_{15}+D_{16}+A_{1}+\ldots +A_{m-3}\textrm{ and }\phi +\mathcal{P}\\
& \{x_{m-1},b_{2},f_{1}\}\to \{x_{m-1},b_{2},f_{2}\}\colon \phi =D_{1}+D_{2}+D_{5}+D_{7}+D_{8}+D_{10}+D_{13}+\\
& +D_{15}+A_{1}+\ldots +A_{m-3}\textrm{ and }\phi +\mathcal{P}
\end{xalignat*}
By an analogous reasoning as in the class $\mathfrak{s}_{0}+\mu _{1}$ we conclude that $HF^{+}(Y,\mathfrak{s}_{0}-\mu _{2})\cong \ta \oplus \ta $ is freely generated by the elements $[\{x_{m-1},b_{2},f_{2}\},i]$ and $[\{x_{m-1},b_{2},f_{1}\},i]$ for $i\geq 0$. 

\noindent \textbf{Class $\mathfrak{s}_{0}-\mu _{1}$}
\begin{displaymath}
\xymatrix{\{x_{m},c_{2},e_{2}\} \ar@{->}[d] & \quad \\
\{u_{1},a_{2},e_{2}\} & \{x_{2},b_{n-1},f_{2}\} \ar@{.>}[d]\\
\quad & \{x_{2},b_{n-1},f_{1}\}}
\end{displaymath}
\begin{xalignat*}{1}
& \{x_{m},c_{2},e_{2}\}\to \{u_{1},a_{2},e_{2}\}\colon D_{15}\, ,\textrm{ rectangle}\\
& \{x_{m},c_{2},e_{2}\}\to \{x_{2},b_{n-1},f_{2}\}\colon -D_{2}-D_{3}-D_{4}-D_{6}+D_{8}+D_{10}+D_{15}+A_{1}+\\
& +A_{2}+\ldots +A_{m-3}\, ,\textrm{ no holomorphic representatives}\\
& \{u_{1},a_{2},e_{2}\}\to \{x_{m},c_{2},e_{2}\}\colon \phi =D_{1}+D_{2}+D_{5}+D_{7}+D_{8}+D_{10}+2D_{13}+A_{1}+\\
& +A_{2}+\ldots +A_{m-3}\;, \phi +\mathcal{P}\textrm{ and }\phi +2\mathcal{P}\\
&  \{u_{1},a_{2},e_{2}\}\to \{x_{2},b_{n-1},f_{1}\}\colon -D_{2}-D_{3}-D_{4}-D_{6}+D_{8}+D_{10}+D_{13}+A_{1}+\\
& +A_{2}+\ldots +A_{m-3}\, ,\textrm{ no holomorphic representatives}\\
& \{x_{2},b_{n-1},f_{2}\}\to \{x_{m},c_{2},e_{2}\}\colon \phi =D_{1}+2D_{2}+D_{3}+D_{4}+D_{5}+D_{6}+D_{7}+2D_{13}\;,\\
& \phi +\mathcal{P}\textrm{ and }\phi +2\mathcal{P}\\
& \{x_{2},b_{n-1},f_{2}\}\to \{x_{2},b_{n-1},f_{1}\}\colon D_{13}\textrm{ and }D_{13}+\mathcal{P}\\
& \{x_{2},b_{n-1},f_{1}\}\to \{u_{1},a_{2},e_{2}\}\colon D_{1}+2D_{2}+2D_{3}+2D_{4}+D_{5}+2D_{6}+D_{7}+D_{9}+\\
& +D_{11}+D_{12}+D_{14}+D_{15}+D_{16}+B_{1}+\ldots +B_{n-3}\\
& \{x_{2},b_{n-1},f_{1}\}\to \{x_{2},b_{n-1},f_{2}\}\colon \phi =D_{1}+D_{2}+D_{5}+D_{7}+D_{8}+D_{10}+D_{13}+\\
& +D_{15}+A_{1}+\ldots +A_{m-3}\textrm{ and }\phi +\mathcal{P}
\end{xalignat*}
Again we use an analogous reasoning as in the class $\mathfrak{s}_{0}+\mu _{1}$ to obtain $$HF^{+}(Y,\mathfrak{s}_{0}-\mu _{1})\cong \ta \oplus \ta \;,$$ freely generated by the elements $[\{x_{2},b_{n-1},f_{2}\},i]$ and $[\{x_{2},b_{n-1},f_{1}\},i]$ for $i\geq 0$. 

In the following calculations, we apply the change of basepoint formula using \cite[Lemma 2.19]{OS1}: 
\begin{lemma}
\label{lemaz}
Let $(\Sigma ,(\alpha _{1},\ldots ,\alpha _{g}),(\beta _{1},\ldots ,\beta _{g}),z_{1})$ be a Heegaard diagram. Denote by $z_{2}\in \Sigma -\alpha _{1}-\ldots -\alpha _{g}-\beta _{1}-\ldots -\beta _{g}$ a new basepoint, for which the following holds: there is an arc $z_{t}$ from $z_{1}$ to $z_{2}$ on the surface $\Sigma $, which is disjoint from all curves $\beta _{i}$ and from all curves $\alpha _{i}$ appart from $\alpha _{j}$. Then for any generator $\mathbf{x}\in \TT _{\alpha }\cap \TT _{\beta }$ we have $$\mathfrak{s}_{z_{2}}(\mathbf{x})-\mathfrak{s}_{z_{1}}(\mathbf{x})=\alpha _{j}^{*}\;,$$ where $\alpha _{j}^{*}\in H^{2}(Y;\ZZ )$ is the Poincar\'{e} dual of the homology class in $\y $ induced by the curve $\gamma $ in $\Sigma $, for which $\alpha _{j}\cdot \gamma =1$ and whose intersection number with any other curve $\alpha _{i}$ for $j\neq i$ equals $0$. 
\end{lemma}

\noindent \textbf{Class $\mathfrak{s}_{0}+\mu _{1}+\mu _{2}$}\\
We change the basepoint $z_{1}\in D_{5}$ for a new basepoint $z_{2}\in D_{2}$. By Lemma \ref{lemaz} the $\sp $ structure $\mathfrak{s}_{0}+\mu _{1}=\mathfrak{s}_{z_{1}}(\{x_{m-2},b_{1},f_{1}\})$ changes to $\mathfrak{s}_{0}+\mu _{1}+\mu _{2}=\mathfrak{s}_{z_{2}}(\{x_{m-2},b_{1},f_{1}\})$. In the new $\sp $ structure we have the same generators as in the class $\mathfrak{s}_{0}+\mu _{1}$, but with a new relative grading, induced by the basepoint $z_{2}$:
\begin{displaymath}
\xymatrix{\quad & \{x_{m-2},b_{1},f_{2}\}\ar@{.>}[d]\\
\{x_{m},c_{1},e_{1}\} \ar@{->}[d] &  \{x_{m-2},b_{1},f_{1}\}\\
\{y_{1},c_{1},d_{1}\} & \quad }
\end{displaymath}
We already know from the class $\mathfrak{s}_{0}+\mu _{1}$ that the only nontrivial differential of $HF^{\infty }$ in this class is $\partial ^{\infty }[\{x_{m},c_{1},e_{1}\},i]=[\{y_{1},c_{1},d_{1}\},i]$. It follows that  $\partial ^{+}[\{x_{m},c_{1},e_{1}\},i]=[\{y_{1},c_{1},d_{1}\},i]$ and also $\partial ^{+}[\{x_{m-2},b_{1},f_{2}\},i]=\partial ^{+}[\{x_{m-2},b_{1},f_{1}\},i]=0$. The resulting homology $$HF^{+}(\y ,\mathfrak{s}_{0}+\mu _{1}+\mu _{2})=\ta \oplus \ta $$ is freely generated by the elements $[\{x_{m-2},b_{1},f_{2}\},i]$ and $[\{x_{m-2},b_{1},f_{1}\},i]$ for $i\geq 0$. 

\noindent \textbf{Class $\mathfrak{s}_{0}-\mu _{1}-\mu _{2}$}\\
We change the basepoint $z_{1}\in D_{5}$ for a new basepoint $z_{2}\in D_{2}$. By Lemma \ref{lemaz}, the $\sp $ structure $\mathfrak{s}_{0}-\mu _{1}-2\mu _{2}=\mathfrak{s}_{z_{1}}(\{x_{m},b_{2},f_{1}\})$ changes to $\mathfrak{s}_{0}-\mu _{1}-\mu _{2}=\mathfrak{s}_{z_{2}}(\{x_{m},b_{2},f_{1}\})$. In the new $\sp $ structure we have the same generators as in the $\sp $ structure $\mathfrak{s}_{0}-\mu _{1}-2\mu _{2}$, but the relative grading is now induced by the basepoint $z_{2}$: 
\begin{displaymath}
\xymatrix{
\{x_{m},b_{2},f_{2}\} \ar@{.>}[d]\\
\{x_{m},b_{2},f_{1}\}}
\end{displaymath}
We already know that in the classes containing only two generators, the differential $\partial ^{\infty }$ is trivial. Therefore the resulting homology is  $$HF^{+}(\y ,\mathfrak{s}_{0}-\mu _{1}-\mu _{2})=\ta \oplus \ta \;,$$ freely generated by the elements $[\{x_{m},b_{2},f_{2}\},i]$ and $[\{x_{m},b_{2},f_{1}\},i]$ for $i\geq 0$. 

\noindent \textbf{Classes $\mathfrak{s}_{0}-i\mu _{1}-\mu _{2}$ for $2\leq i\leq m-2$}\\
We change the basepoint $z_{1}\in D_{5}$ for a new basepoint $z_{2}\in D_{2}$. By Lemma \ref{lemaz}, the $\sp $ structure $\mathfrak{s}_{0}-i\mu _{1}-2\mu _{2}=\mathfrak{s}_{z_{1}}(\{x_{i-1},b_{1},f_{1}\})$ changes to $\mathfrak{s}_{0}-i\mu _{1}-\mu _{2}=\mathfrak{s}_{z_{2}}(\{x_{i-1},b_{1},f_{1}\})$. In the new $\sp $ structure we have the same generators as in the $\sp $ structure $\mathfrak{s}_{0}-i\mu _{1}-2\mu _{2}$, but the relative grading is now induced by the basepoint $z_{2}$: 
\begin{displaymath}
\xymatrix{
\{x_{i-1},b_{1},f_{2}\} \ar@{.>}[d]\\
\{x_{i-1},b_{1},f_{1}\}}
\end{displaymath}
We already know that in the classes containing only two generators, the differential $\partial ^{\infty }$ is trivial. Therefore the resulting homology is  $$HF^{+}(\y ,\mathfrak{s}_{0}-i\mu _{1}-\mu _{2})=\ta \oplus \ta \;,$$ freely generated by the elements $[\{x_{i-1},b_{1},f_{2}\},j]$ and $[\{x_{i-1},b_{1},f_{1}\},j]$ for $j\geq 0$. 

\noindent \textbf{Class $\mathfrak{s}_{0}$}\\
We change the basepoint $z_{1}\in D_{5}$ for a new basepoint $z_{2}\in D_{2}$. By Lemma \ref{lemaz}, the $\sp $ structure $\mathfrak{s}_{0}-\mu _{2}=\mathfrak{s}_{z_{1}}(\{x_{m-1},b_{2},f_{1}\})$ changes to $\mathfrak{s}_{0}=\mathfrak{s}_{z_{2}}(\{x_{m-1},b_{2},f_{1}\})$. In the new $\sp $ structure we have the same generators as in the $\sp $ structure $\mathfrak{s}_{0}-\mu _{2}$, but with a new relative grading, induced by the basepoint $z_{2}$:
\begin{displaymath}
\xymatrix{ \{u_{2},a_{2},e_{1}\} & \{x_{m-1},b_{2},f_{2}\} \ar@{.>}[d]\\
\{x_{1},c_{1},e_{1}\} \ar@{->}[u] & \{x_{m-1},b_{2},f_{1}\}}
\end{displaymath}
From our calculation in the $\sp $ structure $\mathfrak{s}_{0}-\mu _{2}$ we deduce that the only nontrivial differential of $CF^{\infty }$ in this class is $\partial ^{\infty }[\{x_{1},c_{1},e_{1}\},i]=[\{u_{2},a_{2},e_{1}\},i-1]$. In the complex $CF^{+}$ we have $\partial ^{+}[\{x_{1},c_{1},e_{1}\},i]=[\{u_{2},a_{2},e_{1}\},i-1]$ for $i\geq 1$ and $\partial ^{+}[\{x_{1},c_{1},e_{1}\},0]=0$. It follows that 
\begin{xalignat*}{1}
& HF^{+}(\y ,\mathfrak{s}_{0})\cong \ta \oplus \ta \oplus \FF [\{x_{1},c_{1},e_{1}\},0]\;,
\end{xalignat*} where the first two summands are freely generated by the elements $[\{x_{m-1},b_{2},f_{2}\},i]$ and $[\{x_{m-1},b_{2},f_{1}\},i]$ for $i\geq 0$. 
The homology group $HF^{\infty }(\y ,\mathfrak{s}_{0})$ however equals $$HF^{\infty }(\y ,\mathfrak{s}_{0})\cong \FF [U,U^{-1}]\oplus \FF [U,U^{-1}]\,.$$

We have thus calculated the Heegaard--Floer homology $HF^{+}(\y ,\mathfrak{s})$ for all torsion $\sp $ structures $\mathfrak{s}$ on the manifold $\y $. In the following Subsection, we calculate the absolute gradings of the generators and finish the proof of Theorem \ref{th1}.  

If $b_{1}(Y)>0$ then there is an action of the exterior algebra $\Lambda ^{*}(H_{1}(Y;\ZZ )/\operatorname{Tors})$ on the groups $HF^{\infty }(Y,\mathfrak{s})$ and $\widehat{HF}(Y,\mathfrak{s})$ for every torsion $\sp $ structure $\mathfrak{s}$ on $Y$ \cite[Proposition 4.17, Remark 4.20]{OS1}. Let $\gamma $ be a simple closed curve on the Heegaard surface $\Sigma $ in general position with respect to the $\alpha $ curves and let $[\gamma ]$ be its induced homology class in $H_{1}(Y,\ZZ )$. Then the action is given by $$A_{[\gamma ]}([\mathbf{x},i])=\sum _{\mathbf{y}}\sum _{\{\phi \in \pi _{2}(\mathbf{x},\mathbf{y})|\, \mu (\phi )=1\}}a(\gamma ,\phi )\cdot [\mathbf{y},i-n_{z}(\phi )]\;,$$ where $$a(\gamma ,\phi )=\# \{u\in \mathcal{M}(\phi )|\, u(1\times 0)\in (\gamma \times \textrm{Sym}^{g-1}(\Sigma ))\cap \mathbb{T}_{\alpha }\}\;.$$
The value $d_{b}(Y,\mathfrak{s})$ is the least grading of an element of $HF^{\infty }(Y,\mathfrak{s})$ that is in the kernel of the action of $\Lambda ^{*}(H_{1}(Y;\ZZ )/\operatorname{Tors})$ and whose image in $HF^{+}(Y,\mathfrak{s})$ is nonzero. 
 
In the case of $Y=\y $ and $\mathfrak{s}$ any torsion $\sp $ structure on $\y $, the image of $HF^{+}(\y ,\mathfrak{s})$ in $d\widehat{HF}(\y ,\mathfrak{s})$ is generated by two elements of the form $\{x_{r},b_{j},f_{2}\}$ and $\{x_{r},b_{j},f_{1}\}$. As we have shown in the beginning of this subsection, there are two homotopy classes of disks $\phi _{1}$ and $\phi _{2}$ from $\{x_{r},b_{j},f_{2}\}$ to $\{x_{r},b_{j},f_{1}\}$ (represented by the domains $D_{13}$ and $D_{13}+\mathcal{P}$) and they both have an odd number of holomorphic representatives. Thus, we have $\# \widehat{\mathcal{M}}(\phi _{1})=\# \widehat{\mathcal{M}}(\phi _{2})=1$. The group $H_{1}(\y ;\ZZ )/\operatorname{Tors}=\ZZ $ is generated by the simple closed curve $\mu _{3}$ on the Heegaard diagram (see Figure \ref{fig:heeg}), so $a(\gamma ,\phi _{1})=0$ and $a(\gamma ,\phi _{2})=1$. It follows that $$A_{[\gamma ]}([\{x_{r},b_{j},f_{2}\},i])=[\{x_{r},b_{j},f_{1}\},i]$$ and the action on $\{x_{r},b_{j},f_{1}\}$ is trivial. So $d_{b}(\y ,\mathfrak{s})$ is given as the absolute grading of the generator $\{x_{r},b_{j},f_{1}\}$ and $d_{t}(\y ,\mathfrak{s})$ is the absolute grading of the generator $\{x_{r},b_{j},f_{2}\}$. For the definitions of the bottom and top correction terms, see \cite[Definition 3.3]{LRS}. 

\end{subsection}

\begin{subsection}{Absolute gradings}
\label{absolute}

The absolute grading of the generators of $\widehat{HF}(\y )$ can be calculated using the cobordism $W$ from $\y $ to the simpler 3--manifold $-L(m,1)\# S^{1}\times S^{2}$ whose absolute grading is known. To construct the cobordism, we use a pointed Heegaard triple $(\Sigma ,\vec{\alpha },\vec{\beta },\vec{\gamma },z)$. Here the first two sets of the curves $\vec{\alpha },\vec{\beta }$ stay the same as before, so $Y_{\alpha ,\beta }=\y $. The curves $\gamma _{1}$ and $\gamma _{3}$ are parallel copies of the curves $\beta _{1}$ and $\beta _{3}$ respectively, and the curve $\gamma _{2}$ is homologous to the meridian $\mu _{2}$ (see Figure \ref{fig:triple}). This means $Y_{\beta ,\gamma }=\# ^{2}S^{1}\times S^{2}$ and $Y_{\alpha ,\gamma }=-L(m,1)\# S^{1}\times S^{2}$. 

Filling the second boundary component $\# ^{2}S^{1}\times S^{2}$ by $\# ^{2}S^{1}\times B^{3}$ we get the surgery cobordism $W$ from $\y $ to $ -L(m,1)\# S^{1}\times S^{2}$. The cobordism $W$ equipped with a $\sp $ structure $\mathfrak{s}$ induces a map $$F_{W,\mathfrak{s}}\colon \widehat{HF}(\y )\to \widehat{HF}(-L(m,1)\# S^{1}\times S^{2})\;.$$ 
Under this map, the absolute grading of a generator $\zeta \in \widehat{HF}(\y )$ is changed by \cite[Formula (4)]{OS4}: 
\begin{xalignat}{1}
\label{grading}
& \gr (F_{W,\mathfrak{s}}(\zeta ))-\gr (\zeta )=\frac{c_{1}(\mathfrak{s})^{2}-2\chi (W)-3\sigma (W)}{4}\;.
\end{xalignat}
\newpage
\newgeometry{margin=1cm}
\thispagestyle{empty} 
\begin{figure}[here]
\labellist
\scriptsize \hair 2pt
\pinlabel \rotatebox{90}{$\alpha _{1}$} at 365 295
\pinlabel \rotatebox{90}{$\alpha _{2}$} at 376 657
\pinlabel \rotatebox{90}{$\alpha _{3}$} at 350 484
\pinlabel \rotatebox{90}{$\beta _{1}$} at 488 265
\pinlabel \rotatebox{90}{$\beta _{2}$} at 324 520
\pinlabel \rotatebox{90}{$\beta _{3}$} at 183 402
\pinlabel \rotatebox{90}{$\gamma _{1}$} at 520 265
\pinlabel \rotatebox{90}{$\gamma _{2}$} at 336 585
\pinlabel \rotatebox{90}{$\gamma _{3}$} at 152 402
\pinlabel \rotatebox{90}{$x_{1}$} [b] at 288 286
\pinlabel \rotatebox{90}{$x_{2}$} [b] at 304 286
\pinlabel \rotatebox{90}{$x_{m-2}$}  [b] at 332 286
\pinlabel \rotatebox{90}{$x_{m-1}$} [t] at 355 283
\pinlabel \rotatebox{90}{$x_{m}$} [b] at 395 286
\pinlabel \rotatebox{90}{$t_{1}^{+}$} at 362 310
\pinlabel \rotatebox{90}{$t_{1}^{-}$} at 363 342
\pinlabel \rotatebox{90}{$t_{2}^{+}$} at 400 442
\pinlabel \rotatebox{90}{$t_{2}^{-}$} at 374 450
\pinlabel \rotatebox{90}{$f_{1}$} [b] at 370 475
\pinlabel \rotatebox{90}{$f_{2}$} [b] at 402 475
\pinlabel \rotatebox{90}{$f_{1}'$} [b] at 382 475
\pinlabel \rotatebox{90}{$f_{2}'$} [b] at 390 475
\pinlabel \rotatebox{90}{$s$} [b] at 335 652
\pinlabel \rotatebox{90}{$b_{1}$}  [b] at 280 652
\pinlabel \rotatebox{90}{$b_{2}$} [b] at 298 652
\pinlabel \rotatebox{90}{$b_{n-1}$} [b] at 324 652
\pinlabel \rotatebox{90}{$b_{n}$} [b] at 364 652
\pinlabel \rotatebox{90}{$r$} at 208 708
\pinlabel \rotatebox{90}{$\mathcal{D}_{1}$} at 500 324
\pinlabel \rotatebox{90}{$\mathcal{D}_{2}$} at 473 370
\pinlabel \rotatebox{90}{$\mathcal{D}_{3}$} at 457 360
\pinlabel \rotatebox{90}{$\mathcal{D}_{4}$} at 468 530
\pinlabel \rotatebox{90}{$\mathcal{D}_{5}$} at 456 460
\pinlabel \rotatebox{90}{$\mathcal{D}_{6}$} at 422 360
\pinlabel \rotatebox{90}{$\mathcal{D}_{7}$} at 430 432
\pinlabel \rotatebox{90}{$\mathcal{D}_{8}$} at 396 396
\pinlabel \rotatebox{90}{$\mathcal{D}_{8}$} at 204 414
\pinlabel \rotatebox{90}{$\mathcal{D}_{8}$} at 370 214
\pinlabel \rotatebox{90}{$\mathcal{D}_{9}$} at 370 360
\pinlabel \rotatebox{90}{$\mathcal{D}_{10}$} at 353 380
\pinlabel \rotatebox{90}{$\mathcal{D}_{11}$} at 352 576
\pinlabel \rotatebox{90}{$\mathcal{D}_{12}$} at 369 576
\pinlabel \rotatebox{90}{$\mathcal{D}_{13}$} at 381 576
\pinlabel \rotatebox{90}{$\mathcal{D}_{14}$} at 397 576
\pinlabel \rotatebox{90}{$\mathcal{D}_{15}$} at 324 380
\pinlabel \rotatebox{90}{$\mathcal{D}_{16}$} at 406 692
\pinlabel \rotatebox{90}{$\mathcal{D}_{16}$} at 490 214
\pinlabel \rotatebox{90}{$\mathcal{D}_{17}$} at 394 692
\pinlabel \rotatebox{90}{$\mathcal{D}_{17}$} at 456 214
\pinlabel \rotatebox{90}{$\mathcal{D}_{18}$} at 265 788
\pinlabel \rotatebox{90}{$\mathcal{D}_{18}$} at 429 200
\pinlabel \rotatebox{90}{$\mathcal{D}_{19}$} at 372 684
\pinlabel \rotatebox{90}{$\mathcal{D}_{19}$} at 411 191
\pinlabel \rotatebox{90}{$\mathcal{D}_{20}$} at 224 512
\pinlabel \rotatebox{90}{$\mathcal{D}_{20}$} at 396 194
\pinlabel \rotatebox{90}{$\mathcal{D}_{21}$} at 310 490
\pinlabel \rotatebox{90}{$\mathcal{D}_{22}$} at 175 624
\pinlabel \rotatebox{90}{$\mathcal{D}_{23}$} at 240 711
\pinlabel \rotatebox{90}{$\mathcal{D}_{24}$} at 216 684
\pinlabel \rotatebox{90}{$\mathcal{D}_{25}$} at 320 548
\pinlabel \rotatebox{90}{$\mathcal{D}_{26}$} at 260 399
\pinlabel \rotatebox{90}{$\mathcal{D}_{27}$} at 295 455
\pinlabel \rotatebox{90}{$\mathcal{D}_{28}$} at 200 455
\pinlabel \rotatebox{90}{$\mathcal{D}_{28}$} at 380 200
\pinlabel \rotatebox{90}{$\mathcal{D}_{29}$} at 386 445
\pinlabel \rotatebox{90}{$\mathcal{D}_{30}$} at 404 461
\pinlabel \rotatebox{90}{$\mathcal{D}_{31}$} at 390 467
\pinlabel \rotatebox{90}{$\mathcal{D}_{32}$} at 375 468
\pinlabel \rotatebox{120}{$\mathcal{A}_{1}$} at 282 303
\pinlabel \rotatebox{110}{$\mathcal{B}_{1}$} at 294 672
\pinlabel \rotatebox{90}{$\mathcal{D}_{33}$} at 350 307 
\pinlabel \rotatebox{90}{$\mathcal{D}_{34}$} at 365 324
\pinlabel \rotatebox{90}{$\mu _{1}$} at 246 280
\pinlabel \rotatebox{90}{$\mu _{2}$} at 240 648
\pinlabel \rotatebox{90}{$\mu _{3}$} at 275 490
\pinlabel \rotatebox{90}{$z$} at 250 440
\endlabellist
\centering
\begin{center}
\caption{The Pointed Heegaard triple}
\includegraphics[scale=0.90]{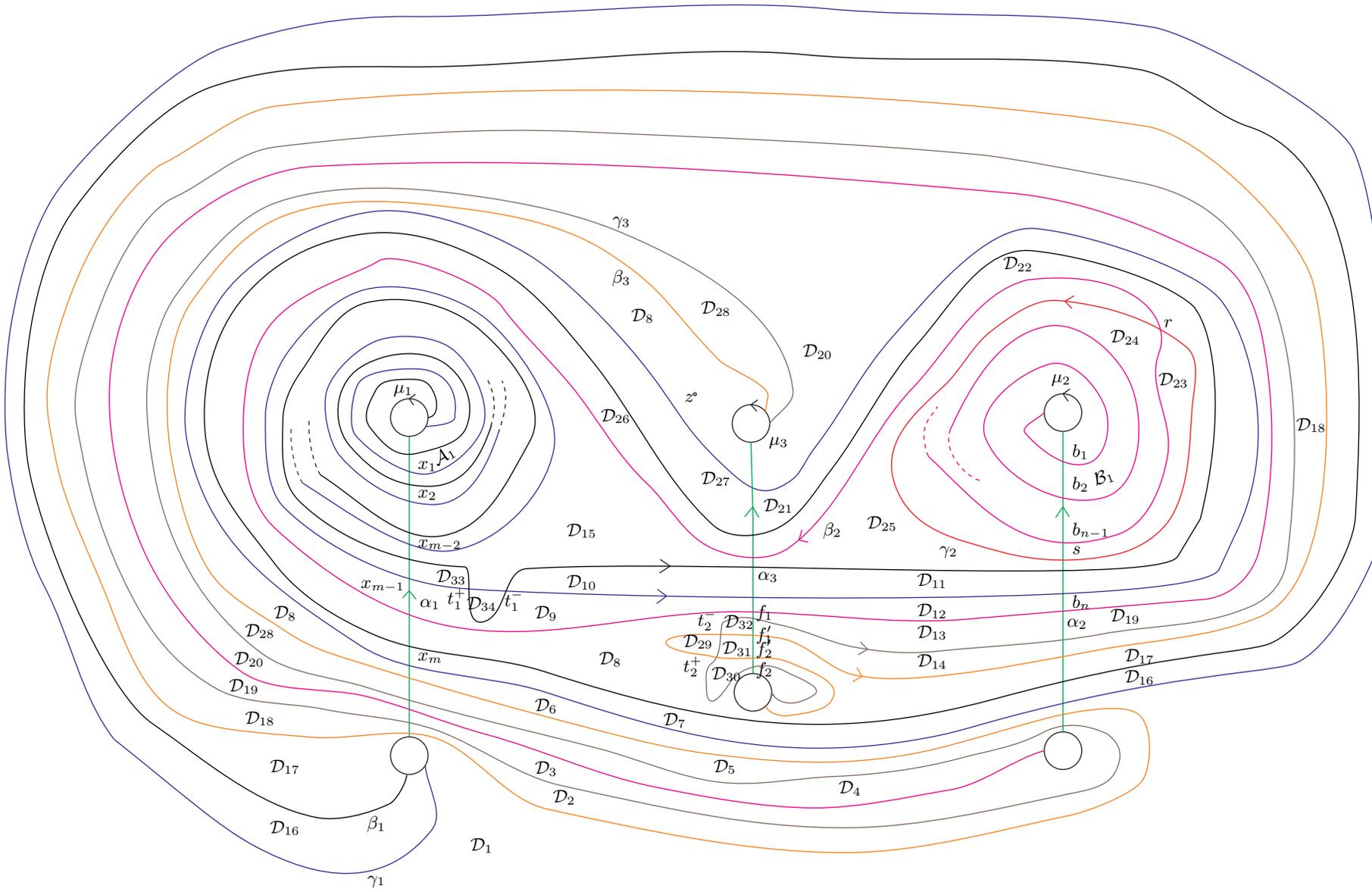}
\label{fig:triple}
\end{center}
\end{figure}
\restoregeometry

The intersections between the $\alpha $ and $\beta $ curves are denoted in the same way as before. New intersections between the $\alpha $, $\beta $ and $\gamma $ curves we will need are denoted by: 
$\alpha _{1}\cap \gamma _{1}=\{x_{1}',x_{2}',\ldots ,x_{m}'\}$, $\alpha _{2}\cap \gamma _{2}=\{s\}$, $\alpha _{3}\cap \gamma _{3}=\{f_{1}',f_{2}'\}$, $\beta _{1}\cap \gamma _{1}=\{t_{1}^{+},t_{1}^{-}\}$, $\beta _{2}\cap \gamma _{2}=\{r\}$, $\beta _{3}\cap \gamma _{3}=\{t_{2}^{+},t_{2}^{-}\}$ (see Figure \ref{fig:triple}). 

We express the $\alpha $, $\beta $ and $\gamma $ curves of the Heegaard triple in the standard basis of the surface $\Sigma $ as:
\begin{xalignat*}{1}
& \alpha _{i}\sim \lambda _{i}\textrm{ for }i=1,2,3\\
& \beta _{1}\sim \gamma _{1}\sim m\mu _{1}+2\mu _{2}-\lambda _{1}\\
& \beta _{2}\sim 2\mu _{1}+n\mu _{2}-\lambda _{2}\\
& \beta _{3}\sim \gamma _{3}\sim \lambda _{3}\\
& \gamma _{2}\sim \mu _{2}
\end{xalignat*} 
The elementary domains in the winding region of the curve $\beta _{1}$ are denoted by $A_{i}$ for $i=1,\ldots ,2m-5$,  the elementary domains in the winding region of the curve $\beta _{2}$ are denoted by $B_{j}$ for $j=1,\ldots ,n-3$ and the other elementary domains of the Heegaard triple are denoted by $D_{i}$ for $i=1,\ldots ,34$. There are four hexagons $D_{1},D_{3},D_{9}$ and $D_{20}$, three pentagons $D_{22}$, $D_{24}$ and $D_{25}$, five triangles $D_{10}$, $D_{23}$, $D_{30}$, $D_{32}$ and $D_{33}$, two bigons $D_{29}$ and $D_{34}$, one octagon $D_{15}$ and a domain $D_{8}$ with 14 sides. All the other elementary domains are rectangles. We put the basepoint into the elementary domain $D_{8}$, which corresponds to the basepoint $z\in D_{5}$ of the Heegaard diagram \ref{fig:heeg}. 

We have a triply-periodic domain
\begin{xalignat*}{1}
& \mathcal{Q}=(m-2)(D_{1}+D_{2}+D_{3})-2(D_{4}+D_{5}+D_{6}+D_{7})-m(D_{9}+D_{10}+D_{11}+D_{12})+\\
& +(2-m)D_{15}+(m-2)D_{16}+m(D_{17}+D_{18}+D_{19})+2D_{22}+(mn-2)D_{23}+\\
& +(m(n-1)-2)D_{24}+(2-m)D_{25}+2D_{26}-mD_{33}+(2-m)D_{34}+\\
& +\sum _{i=1}^{m-3}\left (m-2(i+1))(A_{2i-1}+A_{2i}\right )+\sum _{j=1}^{n-3}\left ((j+1)m-2\right )B_{j}
\end{xalignat*}
The orientation of the curves in the Heegaard triple is denoted on the diagram. The boundary of the triply-periodic domain is equal to $$\partial \mathcal{Q}=2\alpha _{1}+2\beta _{1}-m\alpha _ {2}-m\beta _{2}+(mn-4)\gamma _{2}\;.$$
We calculate the Euler measure of the triply-periodic domain \cite[Lemma 6.2]{OS5}: 
\begin{xalignat*}{1}
& \widehat{\chi}(\mathcal{Q})=2(m-2)(1-\frac{6}{4})-m(1-\frac{6}{4}+1-\frac{3}{4})+(2-m)(1-\frac{8}{4})+2(1-\frac{5}{4})+\\
& +(mn-2)(1-\frac{3}{4})+(m(n-1)-2)(1-\frac{5}{4})+(2-m)(1-\frac{5}{4})-m(1-\frac{3}{4})+\\
& +(2-m)(1-\frac{2}{4})=0\;.
\end{xalignat*}
We have $n_{z}(\mathcal{Q})=0$ and $\# (\partial \mathcal{Q})=m(n+2)$. The self-intersection number $\mathcal{H}(\mathcal{Q})^{2}$ is calculated by counting the intersections of $\alpha $ and $\beta $ curves in the boundary of the triply-periodic domain (according to the chosen orientation of the boundary). We get
$\alpha _{1}\cdot \beta _{1}=-m$, $\alpha _{1}\cdot \beta _{2}=-2$, $\alpha _{2}\cdot \beta _{1}=-2$ and $\alpha _{2}\cdot \beta _{2}=-n$, which gives us 
\begin{xalignat*}{1}
&\mathcal{H}(\mathcal{Q})^{2}=\partial _{\alpha }\mathcal{Q}\cdot \partial _{\beta }\mathcal {Q}=4\alpha _{1}\cdot \beta _{1}-2m\alpha _{1}\cdot \beta _{2}-2m\alpha _{2}\cdot \beta _{1}+m^{2}\alpha _{2}\cdot \beta _{2}=-m(mn-4)
\end{xalignat*}
Since the self-intersection number is negative for $mn-4>0$, the signature of the associated cobordism equals $\sigma (W)=-1$. $W$ is the surgery cobordism from $\y $ to $Y_{\alpha ,\gamma }=L(m,1)\#S^{1}\times S^{2}$, thus $\chi (W)=1$.  

Next we investigate the domains of Whitney triangles on the Heegaard surface. A Whitney triangle connecting $\mathbf{x}$, $\mathbf{y}$ and $\mathbf{w}$ is given by a map $u\colon \Delta \to \operatorname{Sym}^{g}\Sigma $ for which $u(v_{\gamma })=\mathbf{x}$, $u(v_{\alpha })=\mathbf{y}$, $u(v_{\beta })=\mathbf{w}$ and $u(e_{\alpha })\subset T_{\alpha }$, $u(e_{\beta })\subset T_{\beta }$ in $u(e_{\gamma })\subset T_{\gamma }$. The dual spider number of a triangle $u$ and a triply-periodic domain $\mathcal{Q}$ is defined in \cite{OS5} by $$\sigma (u,\mathcal{Q})=n_{u(x)}(\mathcal{Q})+\# (a\cap \partial _{\alpha }'\mathcal{Q})+\# (b\cap \partial _{\beta }'\mathcal{Q})+\# (c\cap \partial _{\gamma }'\mathcal{Q})\;,$$ where $x\in \Delta $ is a chosen point in general position and $a$, $b$, $c$ are chosen paths from $x$ to the respective edges $e_{0}$, $e_{1}$ and $e_{2}$ of the triangle $\Delta $. We show the following:

\begin{lemma} 
\label{triangle}
Let the basepoint of the Heegaard diagram \ref{fig:heeg} lie in the elementary domain $D_{5}$. For $1\leq i\leq m-1$ and $1\leq j\leq n-1$, there is a Whitney triangle $$u\colon \{x_{i},b_{j},f_{2}\}\to \{t_{1}^{+},r,t_{2}^{+}\}\to \{x_{i}',s,f_{2}'\}$$ with $\sigma (u,Q)=-mn+jm-2i$.
\end{lemma}
\begin{proof}
For a Whitney triangle $u\colon \Delta \to \operatorname{Sym}^{3}(\Sigma )$, the image $u(\Delta )$ is a triple branched cover over a triangle. In some cases this is a trivial disconnected cover consisting of three triangles $u_{1}$, $u_{2}$ and $u_{3}$ on the surface $\Sigma $. For $1\leq i\leq m-1$ and $1\leq j\leq n-1$ we can find a triangle with the following components.

The first component is a triangle between the points $x_{i},t_{1}^{+}$ and $x_{i}'$ (for $1\leq i\leq m-1$) with domain $D_{33}+A_{2m-5}+A_{2m-7}+\ldots +A_{2i-1}$ (see Figure \ref{fig:trikot1}). The dual spider number of this component is equal to $\sigma _{1}(u_{i},\mathcal{Q})=m-2(i+1)$. There is also a triangle between the points $x_{m},t_{1}^{+}$ and $x_{m}'$ with the dual spider number $\sigma _{1}(u,\mathcal{Q})=-2$.

The second component of the Whitney triangle (Figure \ref{fig:trikot3}) is a triangle between the points $b_{j}$, $r$ and $s$ (for $1\leq j\leq n-1$) with domain $$(n-j)D_{23}+(n-j-1)D_{24}+(n-j-2)B_{n-3}+(n-j-3)B_{n-4}+\ldots +B_{j}\;,$$ where all the coefficients of the domain have to be positive. The dual spider number of this component is equal to $\sigma _{2}(u_{j},\mathcal{Q})=2-mn+(j-1)m$.

The third component of the Whitney triangle is a triangle between the points $f_{2},t_{2}^{+}$ and $f_{2}'$ with domain $D_{30}$ (Figure \ref{fig:trikot2}). The dual spider number of this component is equal to $\sigma _{3}(u,\mathcal{Q})=0$.  

Combining the above we obtain $$\sigma (u,\mathcal{Q})=\sigma _{1}(u,\mathcal{Q})+\sigma _{2}(u,\mathcal{Q})+\sigma _{3}(u,\mathcal{Q})=-mn+jm-2i\;.$$ 
\end{proof}
\newpage
\thispagestyle{empty}
\newgeometry{margin=1cm}
\begin{figure}[here]
\labellist
\small \hair 2pt
\pinlabel $x_{m-1}$ [r] at 255 725
\pinlabel $\alpha _{1}$ [r] at 255 670
\pinlabel $\beta _{1}$ at 320 730
\pinlabel $\gamma _{1}$ at 320 630
\pinlabel $x_{m-1}'$ [r] at 255 625
\pinlabel $t_{1}'$ at 428 674
\pinlabel $\alpha _{1}$ [r] at 388 370
\pinlabel $\beta _{1}$ at 184 400
\pinlabel $\gamma _{1}$ at 95 216
\pinlabel $x_{m-2}$ [r] at 388 429
\pinlabel $x_{m-2}'$ [r] at 388 314
\pinlabel $t_{1}'$ [l] at 542 216
\pinlabel $-m$ at 432 213
\pinlabel $2-m$ at 132 460
\pinlabel $-m$ [l] at 265 673
\endlabellist
\begin{center}
\includegraphics[scale=0.38]{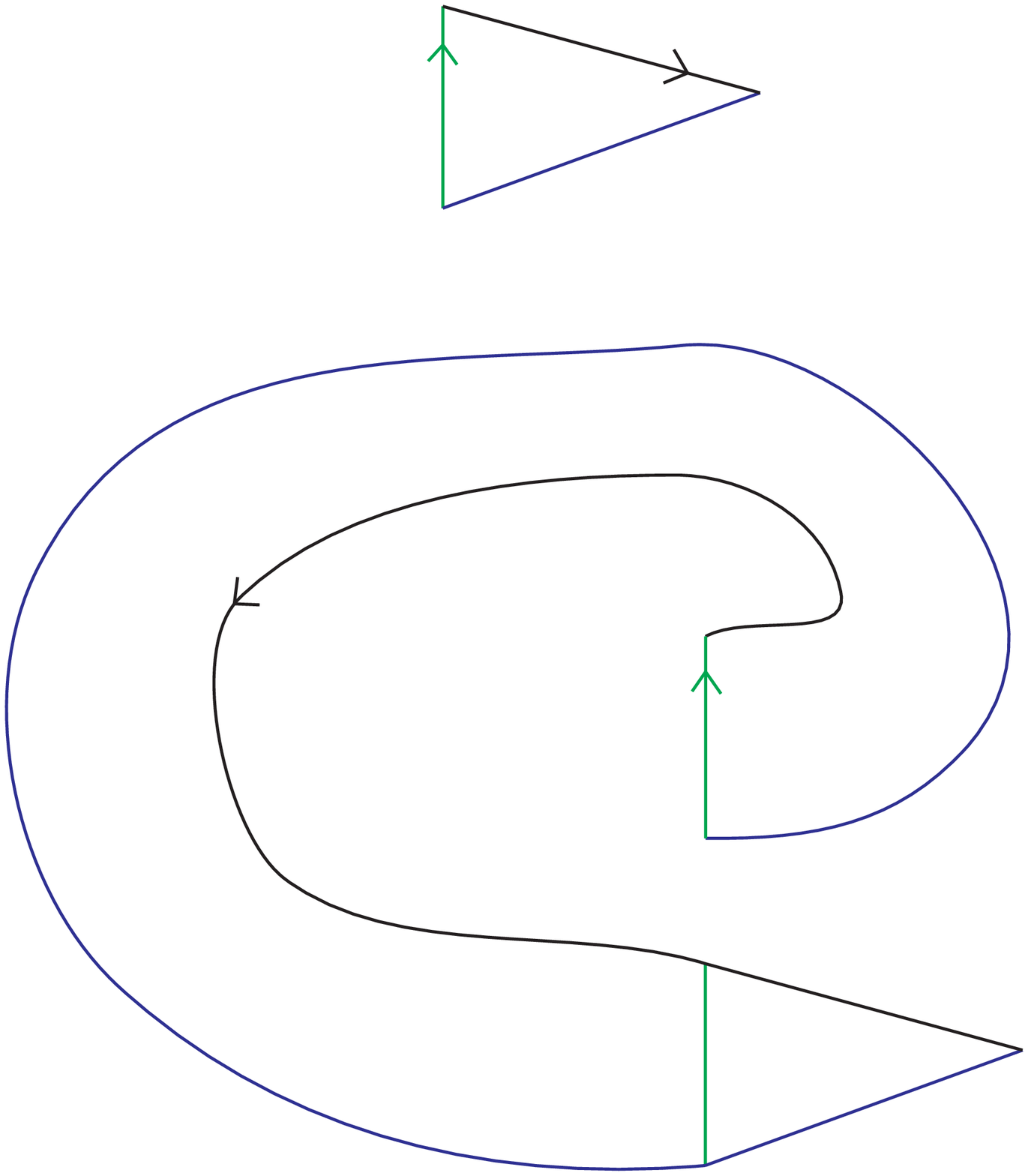}
\caption{The first component of a Whitney triangle; two versions}
\label{fig:trikot1}
\end{center}
\end{figure}
\begin{figure}[here]
\labellist
\small \hair 2pt
\pinlabel $mn-2$ [l] at 230 680
\pinlabel $m(n-1)-2$ at 210 270
\pinlabel $b_{n-1}$ [r] at 210 726
\pinlabel $r$ at 412 800
\pinlabel $\beta _{2}$ at 280 738
\pinlabel $\alpha _{2}$ [r] at 210 672
\pinlabel $\gamma _{2}$ at 390 648
\pinlabel $s$ [r] at 210 622
\pinlabel $\beta _{2}$ at 220 428
\pinlabel $b_{n-2}$ [r] at 324 368
\pinlabel $\alpha _{2}$ [l] at 336 338
\pinlabel $r$ at 536 390
\pinlabel $s$ [r] at 324 204
\pinlabel $\gamma _{2}$ at 440 218
\endlabellist
\begin{center}
\includegraphics[scale=0.38]{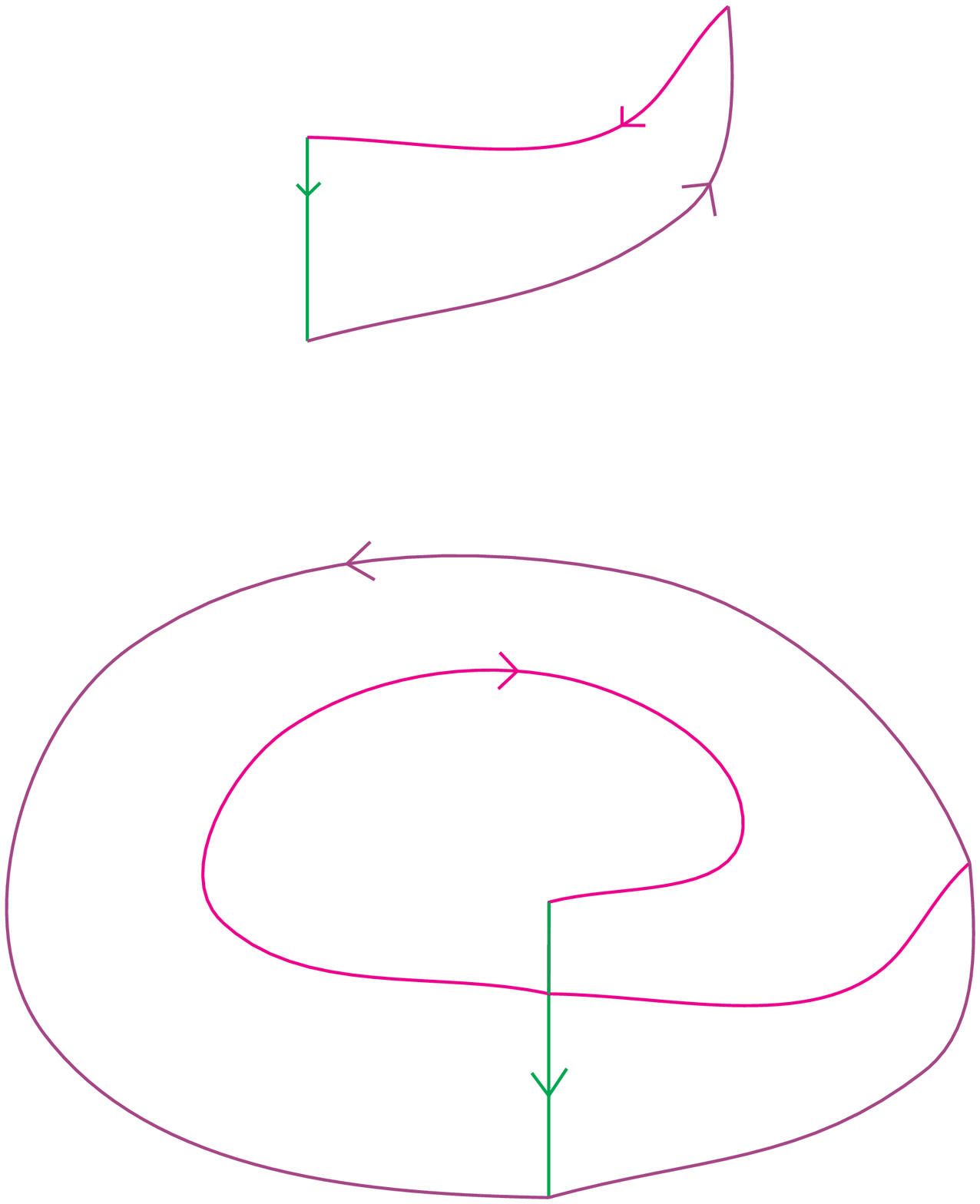}
\caption{The second component of a Whitney triangle; two versions}
\label{fig:trikot3}
\end{center}
\end{figure}
\restoregeometry

\begin{figure}[here]
\labellist
\small \hair 2pt
\pinlabel $\alpha _{3}$ [l] at 680 310
\pinlabel $\beta _{3}$ at 420 460
\pinlabel $\gamma _{3}$ at 400 280
\pinlabel $t_{2}^{+}$ at 160 425
\pinlabel $f_{2}$ [l] at 680 440
\pinlabel $f_{2}'$ [l] at 680 180
\pinlabel $0$ at 560 330
\endlabellist
\begin{center}
\includegraphics[scale=0.30]{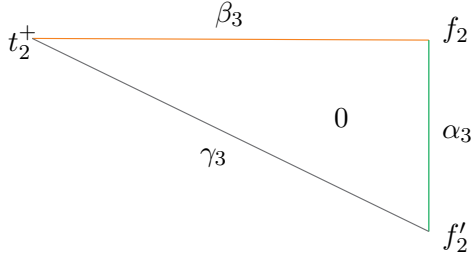}
\label{fig:trikot2}
\caption{The third component of a Whitney triangle}
\end{center}
\end{figure}

We are now prepared to compute the absolute gradings of the generators of $\widehat{HF}(\y )$. 
\begin{proposition}
\label{calc}
If the basepoint of the Heegaard diagram \ref{fig:heeg} lies in the elementary domain $D_{5}$, then the absolute grading of the generator $\{x_{i},b_{j},f_{2}\}$ is given by $$\gr (\{x_{i},b_{j},f_{2}\})=\frac{m^{2}n+mn^{2}-4mn(i+j+1)+4n(i^{2}+2i)+4m(j^{2}+2j)-16ij}{4(mn-4)}$$ for $1\leq i\leq m-1$ and $1\leq j\leq n-1$. 
\end{proposition}
\begin{proof}
By Lemma \ref{triangle}, the generator $\{x_{i},b_{j},f_{2}\}$ is connected to a generator of $\widehat{HF}(-L(m,1)\# S^{1}\times S^{2})$ by a Whitney triangle $$u\colon \{x_{i},b_{j},f_{2}\}\to \{t_{1}^{+},r,t_{2}^{+}\}\to \{x_{i}',s,f_{2}'\}$$ with $\sigma (u,Q)=-mn+jm-2i$. Now we apply the grading shift formula \eqref{grading}. The absolute grading of the generators of $\widehat{HF}(-L(m,1)\# S^{1}\times S^{2})$ can be calculated from \cite[Proposition 4.8]{OS4}. The $i$-th torsion $\sp $ structure on $-L(m,1)\# S^{1}\times S^{2}$ contains two generators: $\{x_{i}',s,f_{2}'\}$ with absolute grading $$\gr (\{x_{i}',s,f_{2}'\})=\frac{(2i-m)^{2}-m}{4m}+\frac{1}{2}$$ and $\{x_{i}',s,f_{1}'\}$ with grading $$\gr (\{x_{i}',s,f_{1}'\})= \frac{(2i-m)^{2}-m}{4m}-\frac{1}{2}$$ where $i=1,\ldots ,m$. We calculate
\begin{xalignat}{1}
\label{c1}
& \left \langle c_{1}(\mathfrak{s}_{z}(u)),\mathcal{H}(\mathcal{Q})\right \rangle =m(n+2)+2\sigma (u,\mathcal{Q})=-mn+2(j+1)m-4i
\end{xalignat}
\begin{xalignat*}{1}
& c_{1}(\mathfrak{s}_{z}(u))^{2}=\frac{\left \langle c_{1}(\mathfrak{s}_{z}(u)),\mathcal{H}(\mathcal{Q})\right \rangle ^{2}}{-m(mn-4)}\\
& \gr (\{x_{i}',s,f_{2}'\})=\frac{(2i-m)^{2}-m}{4m}+\frac{1}{2}=\frac{1}{4}+\frac{(2i-m)^{2}}{4m}\\
& \gr (\{x_{i},b_{j},f_{2}\})=\gr (\{x_{i}',s,f_{2}'\})-\frac{c_{1}(\mathfrak{s}_{z}(u))^{2}-2\chi (W)-3\sigma (W)}{4}=\\
& =\frac{m^{2}n+mn^{2}-4mn(i+j+1)+4n(i^{2}+2i)+4m(j^{2}+2j)-16ij}{4(mn-4)}
\end{xalignat*}  
\end{proof}

Observe the symmetry $\gr (\{x_{m-i},b_{n-j},f_{2}\})=\gr (\{x_{i},b_{j},f_{2}\})$. The above formula calculates the absolute grading  $\gr (\{x_{i},b_{j},f_{2}\})$ for $1\leq i\leq m-1$ and $1\leq j\leq n-1$. 

To calculate the absolute grading of the generators $\{x_{m},b_{j},f_{2}\}$ and $\{x_{i},b_{n},f_{2}\}$ of $\widehat{HF}(\y )$, we use the method of Lee and Lipshitz \cite{LELI}. Their idea is as follows. If two generators $\mathbf{x},\mathbf{y}\in \widehat{HF}(Y)$ represent different torsion $\sp $ structures $\mathfrak{s}_{z}(\mathbf{x})$ and $\mathfrak{s}_{z}(\mathbf{y})$ on a 3-manifold $Y$, then there exists a covering projection $\pi \colon \widetilde{Y}\to Y$ such that $\pi ^{*}\mathfrak{s}_{z}(\mathbf{x})=\pi ^{*}\mathfrak{s}_{z}(\mathbf{y})$  on $\widetilde{Y}$. Thus, there exist lifts $\tilde {\mathbf{x}}$ of $\mathbf{x}$ and $\tilde {\mathbf{y}}$ of $\mathbf{y}$ whose relative grading difference is given by the domain bounded by a closed curve representing $\epsilon (\tilde {\mathbf{x}},\tilde{\mathbf{y}})$. The projection of this domain onto the Heegaard diagram for $Y$ is bounded by some multiple of a closed curve representing $\epsilon (\mathbf{x},\mathbf{y})$. We can reconstruct the relative grading difference between $\mathbf{x}$ and $\mathbf{y}$ from this projection, as described in \cite[Subsection 2.3]{LELI}. 
\begin{proposition}
\label{calc1}
If the basepoint of the Heegaard diagram \ref{fig:heeg} lies in the elementary domain $D_{5}$, then \begin{align}
\label{grmn}
& \gr \{x_{m},b_{j},f_{2}\}=\frac{m^{2}n+mn^{2}-4mnj+4mj^{2}-4m}{4(mn-4)}
\end{align}
\begin{align}
\label{grmn1}
& \gr \{x_{i},b_{n},f_{2}\}=\frac{m^{2}n+mn^{2}-4mni+4ni^{2}-4n}{4(mn-4)}\textrm{  and   }\gr \{x_{m},b_{n},f_{2}\}=\frac{m+n-4}{4} 
\end{align}
for $1\leq i\leq m-1$ and $1\leq j\leq n-1$.
\end{proposition}
\begin{proof}
In the Heegaard diagram \ref{fig:heeg} we find a domain
\begin{align*}
& S=(m+n-4)(D_{1}+D_{16})+(m-2)(D_{2}+D_{3})+(n-2)(-D_{6}-D_{7}+D_{14}+D_{15})+\\
& +(mn-m-n)(D_{8}+D_{9})+(mn-m-2)(D_{10}+D_{11})+\\
& +\sum _{i=1}^{m-3}\left (m+(i+1)n-2(i+2)\right )A_{i}+\sum _{i=1}^{n-3}\left (n+(i+1)m-2(i+2)\right )B_{i}
\end{align*}
for which $\partial \partial _{\alpha }S=(mn-4)(b_{n-1}-b_{n})$. Thus we can compute
\begin{align*}
& \gr \{x_{i},b_{n},f_{k}\}-\gr \{x_{i},b_{n-1},f_{k}\}=\frac{1}{mn-4}\left (e(S)+n_{\{x_{i},b_{n-1},f_{k}\}}(S)+n_{\{x_{i},b_{n},f_{k}\}}(S)\right )=\\
& =\frac{(-mn+4)+(mn-m-n)+2(m+ni-2(i+1))}{mn-4}=\frac{m+(2i-1)n-4i}{mn-4}
\end{align*}
for $1\leq i\leq m-1$ and $\gr \{x_{m},b_{n},f_{k}\}-\gr \{x_{m},b_{n-1},f_{k}\}=\frac{4-m-n}{mn-4}$.

Similarly, the domain 
\begin{align*}
& T=(m+n-4)(D_{1}+D_{16})+(m-2)(D_{2}+D_{3}-D_{10}-D_{11})+\\
& +(mn-n-2)(D_{6}+D_{7})+(mn-m-n)(D_{8}+D_{9})+(n-2)(D_{14}+D_{15})+\\
& +\sum _{i=1}^{m-3}\left (m+(i+1)n-2(i+2)\right )A_{i}+\sum _{i=1}^{n-3}\left (n+(i+1)m-2(i+2)\right )B_{i}
\end{align*}
has $\partial \partial _{\alpha }T=(mn-4)(x_{m-1}-x_{m})$. A calculation gives us
\begin{align*}
& \gr \{x_{m},b_{j},f_{k}\}-\gr \{x_{m-1},b_{j},f_{k}\}=\frac{n+(2j-1)m-4j}{mn-4}
\end{align*}
for $1\leq j\leq n-1$. Combining this with Proposition \ref{calc}, we get formulas \eqref{grmn} and \eqref{grmn1}. 
\end{proof}

In some torsion $\sp $ structures on $\y $ we calculated the homology $HF^{+}(\y )$ by moving the basepoint $z$ into another elementary domain. In those $\sp $ structures we need to perform the calculation of the absolute gradings using the moved basepoint.

\begin{proposition}
\label{special}
Let the basepoint of the Heegaard diagram \ref{fig:heeg} lie in the elementary domain $D_{2}$. Then 
\begin{xalignat*}{1}
& \gr (\{x_{i},b_{j},f_{2}\})=\frac{m^{2}n+mn^{2}-4mn(i+j)+4n(i^{2}+2i)+4m(j^{2}-1)-16i(j-1)}{4(mn-4)}
\end{xalignat*} for $1\leq i\leq m-1$ and $1\leq j\leq n-1$, and 
\begin{xalignat*}{1}
& \gr (\{x_{m},b_{j},f_{2}\})=\frac{m(mn+n^{2}-4nj+4n+4j^{2}-8j)}{4(mn-4)}
\end{xalignat*} for $1\leq j\leq n-1$. 
\end{proposition}
\begin{proof} When calculating $HF^{+}(\y )$ in the $\sp $ structures $\mathfrak{s}_{0}$, $\mathfrak{s}_{0}+\mu _{1}+\mu _{2}$,  $\mathfrak{s}_{0}-\mu _{1}-\mu _{2}$ and $\mathfrak{s}_{0}-i\mu _{1}-\mu _{2}$, we moved the basepoint $z\in D_{5}$ of the basic Heegaard diagram \ref{fig:heeg} over the curve $\alpha _{2}$ into the elementary domain $D_{2}$. Doing the same thing on the triple Heegaard diagram, the basepoint $z\in D_{8}$ moves to $z_{2}\in D_{17}$. The triply-periodic domain $\mathcal{Q}$ now changes to the triply periodic domain $\mathcal{Q}_{2}=\mathcal{Q}-m\Sigma $, for which we have $\partial \mathcal{Q}_{2}=\partial \mathcal{Q}$. As in the previous calculation, we obtain $\# \partial \mathcal{Q}_{2}=m(n+2)$, $n_{{z}_{2}}(\mathcal{Q}_{2})=0$ and $\mathcal{H}(\mathcal{Q}_{2})^{2}=-m(mn-4)$. The Euler measure of the new triply periodic domain is equal to $\widehat{\chi }(\mathcal{Q}_{2})=4m$. We can apply the same Whitney triangles as described in Lemma \ref{triangle}, but now their spider number changes due to the different multiplicities of the elementary domains in $\mathcal{Q}_{2}$. For $1\leq i\leq m-1$ and $1\leq j\leq n-1$, the Whitney triangle $$u\colon \{x_{i},b_{j},f_{2}\}\to \{t_{1}^{+},r,t_{2}^{+}\}\to \{x_{i}',s,f_{2}'\}$$ has the spider number
\begin{xalignat*}{1}
& \sigma (u,\mathcal{Q}_{2})=\sigma _{1}(u,\mathcal{Q}_{2})+\sigma _{2}(u,\mathcal{Q}_{2})+\sigma _{3}(u,\mathcal{Q}_{2})=-2(i+1)+2-(n-j+2)m-m=\\
& =-mn+(j-3)m-2i\;,
\end{xalignat*} while for $i=m$ we have $\sigma _(u,\mathcal{Q}_{2})=-mn+(j-4)m$. Since the basepoint of the triple Heegaard diagram was only moved over the curve $\alpha _{2}$ and not over $\alpha _{1}$, the torsion $\sp $ structures of $-L(m,1)\#S^{1}\times S^{2}$ (and their gradings) remain unchanged. We calculate 
\begin{xalignat*}{1}
& \left \langle c_{1}(\mathfrak{s}_{z}(u)),\mathcal{H}(\mathcal{Q}_{2})\right \rangle =m(n+2)+4m-2mn+2(j-3)m-4i=-mn+2mj-4i\\
& \gr (\{x_{i}',s,f_{2}'\})=\frac{(2i-m)^{2}-m}{4m}+\frac{1}{2}=\frac{1}{4}+\frac{(2i-m)^{2}}{4m}\\
& \gr (\{x_{i},b_{j},f_{2}\})=\gr (\{x_{i}',s,f_{2}'\})-\frac{c_{1}(\mathfrak{s}_{z}(u))^{2}-2\chi (W)-3\sigma (W)}{4}=\\
& =\frac{m^{2}n+mn^{2}-4mn(i+j)+4n(i^{2}+2i)+4m(j^{2}-1)-16i(j-1)}{4(mn-4)}
\end{xalignat*} for $1\leq i\leq m-1$ and $1\leq j\leq n-1$, and 
\begin{xalignat*}{1}
& \gr (\{x_{m},b_{j},f_{2}\})=\frac{m(mn+n^{2}-4nj+4n+4j^{2}-8j)}{4(mn-4)}
\end{xalignat*} for $1\leq j\leq n-1$. 
\end{proof}

\begin{corollary}
\label{specialc}
The absolute grading of the generator $\{x_{m-2},b_{1},f_{2}\}$ in the $\sp $ structure $\mathfrak{s}_{0}+\mu _{1}+\mu _{2}$ is given by 
\begin{align*}
& \gr (\{x_{m-2},b_{1},f_{2}\},z_{2})=\frac{mn(m+n-4)}{4(mn-4)}
\end{align*}
The absolute grading of the generator $\{x_{m-1},b_{2},f_{2}\}$ in the $\sp $ structure $\mathfrak{s}_{0}$ is given by 
\begin{align*}
\gr (\{x_{m-1},b_{2},f_{2}\},z_{2})=\frac{mn(m+n-4)-4(m+n)+16}{4(mn-4)}
\end{align*}
The absolute grading of the generator $\{x_{m},b_{2},f_{2}\}$ in the $\sp $ structure $\mathfrak{s}_{0}-\mu _{1}-\mu _{2}$ is given by 
\begin{align*} 
\gr (\{x_{m},b_{2},f_{2}\},z_{2})=\frac{mn(m+n-4)}{4(mn-4)}
\end{align*}
The absolute grading of the generator $\{x_{i-1},b_{1},f_{2}\}$ in the $\sp $ structure $\mathfrak{s}_{0}-i\mu _{1}-\mu _{2}$ is given by 
\begin{align*}
& \gr (\{x_{i-1},b_{1},f_{2}\},z_{2})=\frac{n(m^{2}+mn-4mi+4i^{2}-4)}{4(mn-4)}
\end{align*}
\end{corollary}
\begin{proof}
We use the formulas from Proposition \ref{special} for the generators $\{x_{m-2},b_{1},f_{2}\}$,  $\{x_{m-1},b_{2},f_{2}\}$, $\{x_{m},b_{2},f_{2}\}$ and $\{x_{i-1},b_{1},f_{2}\}$ to obtain the desired gradings. 
\end{proof}

We have calculated the absolute gradings of the homology generators in the torsion $\sp $ structures on $\y $. Now we identify the $\sp $ structure corresponding to a given generator with a $\sp $ structure $\mathfrak{s}_{i,j}$, defined by \eqref{eqsp}-\eqref{eqsp1}. 
\begin{lemma}
\label{lemmasp} Let the basepoint of the Heegaard diagram \ref{fig:heeg} lie in the elementary domain $D_{5}$. Then 
$$\mathfrak{s}_{i,j}=\mathfrak{s}_{z}(\{x_{i},b_{j},f_{k}\})$$ for $1\leq i\leq m-1$ and $1\leq j\leq n-1$, where $k\in \{1,2\}$. 
\end{lemma}
\begin{proof}
We will show that the two $\sp $ structures are both restrictions of the same $\sp $ structure on the cobordism $W$ from $\y $ to $-L(m,1)\# S^{1}\times S^{2}$. In Lemma \ref{triangle} we described a Whitney triangle $$u\colon \{x_{i},b_{j},f_{2}\}\to \{t_{1}^{+},r,t_{2}^{+}\}\to \{x_{i}',s,f_{2}'\}$$ defining a $\sp $ structure $\mathfrak{s}_{z}(u)$ on $W$ for which $\mathfrak{s}_{z}(u)|_{-L(m,1)\# S^{1}\times S^{2}}=\mathfrak{s}_{z}(\{x_{i}',s',f_{k}'\})$ represents the $i$-th $\sp $ structure on $-L(m,1)\# S^{1}\times S^{2}$ as defined by Ozsv\'ath-Szab\'o in \cite[Subsection 4.1]{OS4}. On the other hand, $\mathfrak{s}_{z}(u)|_{\y }=\mathfrak{s}_{z}(\{x_{i},b_{j},f_{k}\})$.  

Recall that $\mathfrak{s}_{i,j}=\mathfrak{t}_{i,j}|_{\y }$, where $\mathfrak{t}_{i,j}$ is the $\sp $ structure on the manifold $\n $ defined by the Equations \eqref{eqsp}-\eqref{eqsp1}. Since the homology group $H_{2}(\n )=\ZZ ^{2}$ is generated by the base spheres $s_{1}$ and $s_{2}$ of the plumbing $\n $, the $\sp $ structures $\mathfrak{t}_{i,j}$ are well defined. The Kirby diagram of $\n $ on the Figure \ref{fig:kirby} describes the surgery cobordism from $S^{3}$ to the 3-manifold $\y $.  In the first step of the surgery cobordism, we add a 1-handle and a 2-handle along the unknot $K_{1}$ to $S^{3}$, obtaining the 3-manifold $-L(m,1)\# S^{1}\times S^{2}$. The core of the 2-handle union the disk spanned by $K_{1}$ in $B^{4}$ represent the base sphere $s_{1}$. Since by definition $$\langle c_{1}(\mathfrak{t}_{i,j}),s_{1}\rangle =2i-m\;,$$ the restriction $\mathfrak{t}_{i,j}|_{-L(m,1)\# S^{1}\times S^{2}}$ is exactly the $i$-th $\sp $ structure on $-L(m,1)\# S^{1}\times S^{2}$ as defined by Ozsv\'ath-Szab\'o in \cite[Subsection 4.1]{OS4}. Thus, $$\mathfrak{t}_{i,j}|_{-L(m,1)\# S^{1}\times S^{2}}=\mathfrak{s}_{z}(\{x_{i}',s',f_{k}'\})=\mathfrak{s}_{z}(u)|_{-L(m,1)\# S {1}\times S^{2}}\;.$$

The second step of the surgery is given by the cobordism $-W$ from $-L(m,1)\# S^{1}\times S^{2}$ to $\y $. The cobordism $-W$ is given by adding a 2-handle to the boundary of the previously constructed manifold. Let us find a generator of the homology group $H_{2}(-W)=\ZZ $. Writing down the intersection form $$Q_{\n }=\left (\begin{array}{cc}
m & 2\\
2 & n \\
\end{array}\right )$$ for $\n $ and denoting by $F=as_{1}+bs_{2}$ the generator of $H_{2}(-W)$, we use the fact that $F$ has to be orthogonal to the sphere $s_{1}$. Thus, $\langle as_{1}+bs_{2},s_{1}\rangle =ma+2b=0$ and we can take $F=2s_{1}-ms_{2}$. We calculate
\begin{align*}
& \langle c_{1}(\mathfrak{t}_{i,j}),F\rangle =2(2i-m)-m(2j-n)=mn-2(j+1)m+4i\\
& F^{2}=4s_{1}^{2}-4ms_{1}s_{2}+m^{2}s_{2}^{2}=m(mn-4)
\end{align*}  
The first Chern class $c_{1}(\mathfrak{s}_{z}(u))$ of the triangle $$u\colon \{x_{i},b_{j},f_{2}\}\to \{t_{1}^{+},r,t_{2}^{+}\}\to \{x_{i}',s,f_{2}'\}$$ from the Heegaard triple diagram had the same evaluation on the generator $\mathcal{H}(\mathcal{Q})$ of $H_{2}(W)$ (with the opposite sign because of the opposite orientation of the cobordism), see Equation \eqref{c1}. Since also $F^{2}=\mathcal{H}(\mathcal{Q})^{2}$, it follows that the $\sp $ structures coincide on $W$: $\mathfrak{s}_{z}(u)=\mathfrak{t}_{i,j}|_{W}$. Now we have $\mathfrak{s}_{z}(u)|_{\y }=\mathfrak{s}_{z}(\{x_{i},b_{j},f_{k}\})$ and $\mathfrak{t}_{i,j}|_{\y }=\mathfrak{s}_{i,j}$, which gives us the desired equality.  
\end{proof}

\begin{corollary} 
\label{cor1}
Let the basepoint of the Heegaard diagram \ref{fig:heeg} lie in the elementary domain $D_{5}$. Then 
\begin{align*}
& \mathfrak{s}_{0,j}=\mathfrak{s}_{z}(\{x_{m},b_{j+1},f_{k}\})\\  
& \mathfrak{s}_{i,0}=\mathfrak{s}_{z}(\{x_{i+1},b_{n},f_{k}\})
\end{align*}
for $0\leq i\leq m-2$, $0\leq j\leq n-2$ and $k\in \{1,2\}$. 
\end{corollary}
\begin{proof}
We use \cite[Lemma 2.19]{OS1} to evaluate the cohomology class in $H^{2}(\y )$ corresponding to the difference of two $\sp $ structures. We calculate
\begin{align*}
& \mathfrak{s}_{i,j}-\mathfrak{s}_{i+1,j}=\mathfrak{s}_{z}(\{x_{i},b_{j},f_{k}\})-\mathfrak{s}_{z}(\{x_{i+1},b_{j},f_{k}\})=PD[\mu _{1}]\\
& \mathfrak{s}_{i,j}-\mathfrak{s}_{i,j+1}=\mathfrak{s}_{z}(\{x_{i},b_{j},f_{k}\})-\mathfrak{s}_{z}(\{x_{i},b_{j+1},f_{k}\})=PD[\mu _{2}]
\end{align*} 
and by linearity it follows that $\mathfrak{s}_{i,j}+aPD[\mu _{1}]+bPD[\mu _{2}]=\mathfrak{s}_{i-a,j-b}$. Thus
\begin{align*}
& \mathfrak{s}_{z}(\{x_{m},b_{j+1},f_{k}\})=\mathfrak{s}_{z}(\{x_{1},b_{j+1},f_{k}\})+PD[\mu _{1}+\mu _{2}]=\mathfrak{s}_{0,j}\\
& \mathfrak{s}_{z}(\{x_{i+1},b_{n},f_{k}\})=\mathfrak{s}_{z}(\{x_{i+1},b_{1},f_{k}\})+PD[\mu _{1}+\mu _{2}]=\mathfrak{s}_{i,0}
\end{align*} for $0\leq i\leq m-2$, $0\leq j\leq n-2$ and $k\in \{1,2\}$.
\end{proof}

We have thus obtained:

\begin{proof}[Proof of Theorem \ref{th1}]
In Subsection \ref{CF} we have shown that $HF^{+}(\y ,\mathfrak{s})$ has two $\ta $ summands in each torsion $\sp $ structure $\mathfrak{s}$ on $\y $. In one torsion $\sp $ structure, $HF^{+}(\y ,\mathfrak{s})$ has an additional $\FF $ sumand. We have also shown that the action of $\Lambda ^{*}(H_{1}(Y,\ZZ )/\operatorname{Tors})$ maps the generator of $\ta $ with the higher absolute grading to the generator with the lower absolute grading. In Proposition \ref{calc} we have calculated that 
\begin{xalignat*}{1}
& \gr (\{x_{i},b_{j},f_{2}\})=\\
& =\frac{m^{2}n+mn^{2}-4mn(i+j+1)+4n(i^{2}+2i)+4m(j^{2}+2j)-16ij}{4(mn-4)}=d(i,j)
\end{xalignat*} for $1\leq i\leq m-1$ and $1\leq j\leq n-1$. By Lemma \ref{lemmasp}, for those indices we have $\mathfrak{s}_{i,j}=\mathfrak{s}_{z}(\{x_{i},b_{j},f_{k}\})$.

The $\sp $ structures $\mathfrak{s}_{0,j}$ for $0\leq j\leq n-2$ and $\mathfrak{s}_{i,0}$ for $0\leq i\leq m-2$ are identified with the generators of $\widehat{HF}(\y )$ in the Corollary \ref{cor1}, and the absolute grading of those generators has been calculated in Proposition \ref{calc1}. 

The absolute gradings of the generators in the $\sp $ structures $\mathfrak{s}_{0}$, $\mathfrak{s}_{0}+\mu _{1}+\mu _{2}$, $\mathfrak{s}_{0}-\mu _{1}-\mu _{2}$ and $\mathfrak{s}_{0}-i\mu _{1}-\mu _{2}$ are given in Corollary \ref{specialc}. By Lemma \ref{lemmasp}, Corollary \ref{cor1} and Corollary \ref{unique} we have $\mathfrak{s}_{0}=\mathfrak{s}_{1,n-1}=\mathfrak{s}_{m-1,1}$, $\mathfrak{s}_{0}+\mu _{1}+\mu _{2}=\mathfrak{s}_{0,n-2}=\mathfrak{s}_{m-2,0}$, $\mathfrak{s}_{0}-\mu _{1}-\mu _{2}=\mathfrak{s}_{0,0}$ and $\mathfrak{s}_{0}-i\mu _{1}-\mu _{2}=\mathfrak{s}_{i-1,0}$. By Corollary \ref{specialc} we can ascertain that the top correction terms in those $\sp $ structures are given by 
\begin{xalignat*}{1}
& d_{t}(\y ,\mathfrak{s}_{1,n-1})=d(1,n-1)\\
& d_{t}(\y ,\mathfrak{s}_{0,n-2})=d_{1}(m,n-1)\\
& d_{t}(\y ,\mathfrak{s}_{0,0})=d_{1}(m,1)\\
& d_{t}(\y ,\mathfrak{s}_{i-1,0})=d_{1}(n,i)
\end{xalignat*}
It follows that $d_{t}(\y ,\mathfrak{s}_{i,j})=d(i,j)$ for $1\leq i\leq m-1$ and $1\leq j\leq n-1$. Moreover, $d_{t}(\y ,\mathfrak{s}_{0,j})=d_{1}(m,j+1)$ for $0\leq j\leq n-2$ and $d_{t}(\y ,\mathfrak{s}_{i,0})=d_{1}(n,i+1)$ for $0\leq i\leq m-2$. 
\end{proof}

\end{subsection}
\end{section}

\begin{section}{An application} 
\label{App}
Let $X$ be a closed smooth 4--manifold with $H_{1}(X)=0$ and $b_{2}^{+}(X)=2$. Consider two classes $\alpha ,\beta \in H_{2}(X;\ZZ )$ for which the following holds: 
\begin{xalignat*}{1}
& \alpha \cdot \beta =2\\
& \alpha ^{2}=m>0\\
& \beta ^{2}=n>0 \\
& mn-4>0
\end{xalignat*}
Thus the restriction $Q_{X}|_{\ZZ \alpha +\ZZ \beta }$ of the intersection form $Q_{X}$ to the sublattice spanned by $\alpha $ and $\beta $ is positive definite. 

The classes $\alpha $ and $\beta $ can be represented by embedded surfaces $\Sigma _{1},\Sigma _{2}\subset X$ meeting transversally. Suppose that it is possible to choose $\Sigma _{1}$ and $\Sigma _{2}$ to be spheres whose geometric intersection number is 2. Then the regular neighborhood of the union $\Sigma _{1}\cup \Sigma _{2}$ is a double plumbing of disk bundles over spheres $\n $ with boundary $\y $ that has been the object of our investigation in the previous section. The submanifold $\n \subset X$ carries the positive part of the intersection form $Q_{X}$. Denote by $W=X\backslash \operatorname{Int}(\n )$ its complement in $X$. Thus $W$ is a 4--manifold with boundary $-\y $ which carries the negative part of the intersection form $Q_{X}$. The following result \cite[Theorem 9.15]{OS4} describes the constraints given by the $\sp $ structures on $W$ which restrict to a given $\sp $ structure on $-\y $.  
\begin{theorem}
\label{inequality}
Let $Y$ be a three-manifold with standard $HF^{\infty }$, equipped with a torsion $\sp $ structure $\mathfrak{t}$, and let $d_{b}(Y, \mathfrak{t})$ denote its bottom-most correction term, i.e. the one
corresponding to the generator of $HF^{\infty }(Y, t)$ which is in the kernel of the action by $H_{1}(Y)$. Then, for each negative semi-definite four-manifold $W$ which bounds $Y$ so that
the restriction map $H^{1}(W;\ZZ ) \rightarrow H^{1}(Y;\ZZ )$ is trivial, we have the inequality: 
\begin{xalignat}{1}
\label{ineq}
& c_{1}(\mathfrak{s})^{2} + b_{2}^{-}(W) \leq 4d_{b}(Y,\mathfrak{t}) + 2b_{1}(Y)
\end{xalignat}
for all $\sp $ structures $\mathfrak{s}$ over $W$ whose restriction to $Y$ is $\mathfrak{t}$.
\end{theorem}
According to \cite[Theorem 10.1]{OS2}, every 3--manifold $Y$ with $b_{1}(Y)=1$ has standard $HF^{\infty }$. Theorem \ref{inequality} can thus be applied in our case for the pair $(W,-\y )$. Correction terms of the manifold $\y $ have been calculated in the previous section. In order to apply inequality \eqref{ineq}, we have to identify the restriction map $H^{2}(W)\to H^{2}(-\y )$ and see how $\sp $ structures on $W$ restrict to $\sp $ structures on $-\y $. Before considering particular cases we establish the following:
\begin{proposition}
\label{prop1}
With notation as above, $H^{1}(W)=0$, $H^{2}(\n )\cong \ZZ ^{2}$, $H^{2}(W)\cong \ZZ ^{b_{2}^{-}(X)+1}\oplus \tau $ and $H^{2}(\y )\cong \ZZ \oplus T$, where $\tau $ and $T$ are torsion groups and $T$ has order $mn-4$. In the special case when $b_{2}^{-}(X)=0$, we have $T/\tau \cong \tau $.
\end{proposition}
\begin{proof}
Consider the Mayer--Vietoris sequence in cohomology of the triple $(X,\n ,W)$ (all coefficients will be $\ZZ $ unless stated otherwise):
\begin{eqnarray*}
& 0\rightarrow H^{1}(W)\oplus H^{1}(\n )\stackrel{f_{1}}\rightarrow H^{1}(\y ) \stackrel{f_{2}}\rightarrow  H^{2}(X)\stackrel{f_{3}}\rightarrow H^{2}(W)\oplus H^{2}(\n ) \stackrel{f_{4}}\rightarrow H^{2}(\y )\rightarrow 0\\
& 0\rightarrow H^{1}(W)\oplus \ZZ \stackrel{f_{1}}\rightarrow \ZZ \stackrel{f_{2}}\rightarrow \ZZ ^{b_{2}^{-}(X)+2}\stackrel{f_{3}}\rightarrow H^{2}(W)\oplus \ZZ ^{2} \stackrel{f_{4}}\rightarrow \ZZ \oplus T\rightarrow 0
\end{eqnarray*}  
At the beginning and the end of the sequence we have zeros since $H_{1}(X)=0$. Since $H_{1}(\y )\cong \ZZ [\mu _{3}]\oplus T[\mu _{1},\mu _{2}]$, it follows from Poincar\'{e} duality and the universal coefficient theorem that $H^{2}(\y )\cong \ZZ \oplus T$ and $H^{1}(\y )\cong \ZZ $. The torsion ele\-ments $\mu _{1}$ and $\mu _{2}$ are the boundary circles of the fibre disks in the plumbing $\n $. The generator $\mu _{3}$ of the free part comes from the 1-handle of the plumbing, which means that $f_{1}|_{H^{1}(\n )}\colon H^{1}(\n )\to H^{1}(\y )$ is an isomorphism. Thus $H^{1}(W)=0$ and  the restriction map $f_{1}|_{H^{1}(W)}\colon H^{1}(W) \rightarrow H^{1}(\y )$ is always trivial, satisfying the assumption in Theorem \ref{inequality}. Since $f_{1}$ is an isomorphism, by exactness $f_{2}$ is a trivial map. It follows that $f_{3}$ is injective. To understand the homomorphism $f_{4}$, recall the long exact sequence in homology of the pair $(\n ,\y )$:
\begin{eqnarray}
\label{NY}
& \ldots \rightarrow H_{2}(\n )\stackrel{A}\rightarrow H_{2}(\n ,\y )\stackrel{B}\rightarrow H_{1}(\y )\stackrel{C}\rightarrow H_{1}(\n )\rightarrow H_{1}(\n ,\y ) \\
& \ldots \longrightarrow \ZZ ^{2}\stackrel{A}\longrightarrow \ZZ ^{2}\stackrel{B}\longrightarrow \ZZ \oplus T\stackrel{C}\longrightarrow \ZZ \longrightarrow 0
\end{eqnarray}
As described above, the restriction $C|_{\ZZ }\colon \ZZ [\mu _{3}]\to H_{1}(\n )$ is an isomorphism. It follows that the image of the map $B\colon H_{2}(\n ,\y )\to H_{1}(\y )$ is equal to $T$. The same is true for the Poincar\'{e} dual map $f_{4}|_{H^{2}(\n )}\colon H^{2}(\n )\to H^{2}(\y )$ in the Mayer--Vietoris sequence above. So there must be a free sumand $\ZZ \subseteq H^{2}(W)$ which is mapped by $f_{4}$ isomorphically onto the free sumand of $H^{2}(\y )$ (this is the part dual to the part of $H_{2}(W)$ which comes from the boundary). Now since $f_{3}$ is injective, the free subgroup $\ZZ ^{b_{2}^{-}(X)}\subseteq H^{2}(X)$ maps into $H^{2}(W)$ and it follows that the free part of $H^{2}(W)$ has dimension $b_{2}^{-}(X)+1$. Since $H^{1}(W)=0$, it follows from the universal coefficient theorem that $H_{1}(W)=\tau $ is torsion and consequently $$H^{2}(W)\cong \ZZ ^{b_{2}^{-}(X)+1}\oplus \tau \;.$$
Based on our conclusions above, a part of the cohomology Mayer--Vietoris sequence of the triple $(X,\n ,W)$ looks like 
\begin{xalignat}{1}
\label{MV}
\ldots \stackrel{0}\longrightarrow H^{2}(X)\stackrel{f_{3}}\longrightarrow H^{2}(W)\oplus H^{2}(\n )\stackrel{f_{4}}\longrightarrow H^{2}(\y )\longrightarrow 0\\
\ldots \stackrel{0}\longrightarrow \ZZ ^{b_{2}^{-}(X)+2}\stackrel{f_{3}}\longrightarrow (\ZZ ^{b_{2}^{-}(X)+1}\oplus \tau )\oplus \ZZ ^{2}\stackrel{f_{4}}\longrightarrow \ZZ \oplus T\longrightarrow 0
\end{xalignat}  
The restriction $f_{4}|_{H^{2}(\n )}$ can be described by its Poincar\'{e} dual $B\colon H_{2}(\n ,\y )\to H_{1}(\n )$ in the long exact sequence \eqref{NY}. Consider now the restriction $f_{4}|_{H^{2}(W)}$ in \eqref{MV}. The sumand $\tau \subseteq H^{2}(W)$ maps by $f_{4}$ injectively into the torsion group $T\subseteq H^{2}(\y )$. We can observe the Poincar\'{e} dual of the restriction $f_{4}|_{H^{2}(W)}$ in the long exact sequence of the pair $(W,\y )$:
\begin{eqnarray*}
& H_{3}(W,\y )\rightarrow H_{2}(\y )\stackrel{g_{1}}\longrightarrow H_{2}(W)\stackrel{g_{2}}\longrightarrow H_{2}(W,\y )\stackrel{g_{3}}\longrightarrow H_{1}(\y )\stackrel{g_{4}}\longrightarrow H_{1}(W)\rightarrow \ldots \\
& 0\rightarrow \ZZ \stackrel{g_{1}}\longrightarrow \ZZ ^{b_{2}^{-}(X)+1}\stackrel{g_{2}}\longrightarrow \ZZ ^{b_{2}^{-}(X)+1}\oplus \tau \stackrel{g_{3}}\longrightarrow \ZZ \oplus T\stackrel{g_{4}}\longrightarrow \tau \stackrel{0}\rightarrow \ldots 
\end{eqnarray*}
Since $H_{3}(W,\y )\cong H^{1}(W)=0$, the map $g_{1}$ is injective. The homomorphism $g_{2}\colon H_{2}(W)\to H_{2}(W,\y )$ is given by the intersection form $Q_{W}$ of the manifold $W$. $Q_{W}$ is trivial on the sumand $\ZZ \subseteq H_{2}(W)$ which corresponds to the image of $g_{1}$. The restriction $Q_{W}|_{\ZZ ^{b_{2}^{-}(X)}}$ is negative definite. The map $g_{3}$ maps the free sumand of $H_{2}(W,\y )$ which comes from the boundary isomorphically onto the free sumand of $H_{1}(\y )$. In the special case when $b_{2}^{-}(X)=0$, the intersection form $Q_{W}$ is trivial and from the exact sequence above it follows that 
\begin{xalignat}{1}
\label{b20}
& T/\tau \cong \tau \quad \quad \textrm{(when $b_{2}^{-}(X)=0$)}\;.
\end{xalignat} 

\end{proof}

We have described the map $f_{4}$ in the Mayer--Vietoris sequence \eqref{MV} which tells us how cohomology classes on $W$ and $\n $ restrict to cohomology classes on the boundary $\y $. As remarked in Subsubsection \ref{notspin}, $\sp $ structures on 3-- and 4--manifolds may be identified by cohomology classes. Using this identification we may study the restrictions of $\sp $ structures on $W$ and $\n $ to $\sp $ structures on the boundary $\y $. 

When the 4-manifold $X$ has $b_{2}^{-}(X)=0$, the obstruction Theorem \ref{inequality} implies the following result.    
\begin{proposition} 
\label{prop2}
Let $X$ be a closed smooth 4-manifold with $H_{1}(X)=0$, $b_{2}^{+}(X)=2$ and $b_{2}^{-}(X)=0$. Suppose there are two spheres $\Sigma _{1},\Sigma _{2}\subset X$ with $\Sigma _{1} ^{2}=m$, $\Sigma _{2}^{2}=n$ and $\Sigma _{1}\cdot \Sigma _{2}=2$. Denote by $\y $ the boundary of a regular neighbourhood of $\Sigma _{1}\cup \Sigma _{2}$ and let $T=H_{1}(\y )$. Then for some subgroup $\tau \subset T$ with $|\tau |^{2}=|T|$ and some $\sp $ structure $\mathfrak{s}_{0}$ on $\y $, we have $d_{b}(\y ,\mathfrak{s}_{0}+\phi )=-\frac{1}{2}$ for every $\phi \in \tau$. 
\end{proposition}
\begin{proof}
Denote as usual by $\n \subset X$ the regular neighbourhood of $\Sigma _{1}\cup \Sigma _{2}$ and by $W=X\backslash \operatorname{Int}(\n )$ its complement. It follows from Proposition \ref{prop1} that $H^{2}(W)\cong \ZZ \oplus \tau $ for some torsion group $\tau \subset T$ and that $T/\tau \cong \tau $, thus $|T|=|\tau |^{2}$. Recall the Mayer--Vietoris sequence of the triple $(X,\n ,W)$ we discussed in Proposition \ref{prop1}:
\begin{eqnarray*}
\ldots \stackrel{0}\longrightarrow H^{2}(X)\stackrel{f_{3}}\longrightarrow H^{2}(W)\oplus H^{2}(\n )\stackrel{f_{4}}\longrightarrow H^{2}(\y )\longrightarrow 0\\
\ldots \stackrel{0}\longrightarrow \ZZ ^{2}\stackrel{f_{3}}\longrightarrow (\ZZ \oplus \tau )\oplus \ZZ ^{2}\stackrel{f_{4}}\longrightarrow \ZZ \oplus T\longrightarrow 0
\end{eqnarray*}
The $\sp $ structures on $W$ which restrict to the $\sp $ structures on $-\y $ co\-rres\-pond to the image $f_{4}(\tau )\subset T$. Since $b_{2}^{-}(X)=0$, the intersection form $Q_{W}$ of the manifold $W$ is trivial and thus $c_{1}(\mathfrak{s})^{2}=0$ for any $\sp $ structure $\mathfrak{s}$ on the manifold $W$. From Theorem \ref{inequality} it follows that if indeed $-\y $ bounds a negative semi-definite submanifold $W$ inside $X$, then the inequality $d_{b}(-\y ,\mathfrak{t})\geq -\frac{1}{2}$ holds for any torsion $\sp $ structure $\mathfrak{t}$ on $-\y $ which is a restriction of a $\sp $ structure on $W$. The bottom and top correction terms are defined in \cite[Definition 3.3]{LRS}, where also the duality $d_{b}(-\y ,\mathfrak{t})=-d_{t}(\y ,\mathfrak{t})$ is shown \cite[Proposition 3.7]{LRS}. By Theorem \ref{th1} we have $d_{t}(\y ,\mathfrak{t})=d_{b}(\y ,\mathfrak{t})+1$. So for any such $\sp $ structure we have $d_{b}(-\y ,\mathfrak{t})=-d_{t}(\y ,\mathfrak{t})=-d_{b}(\y ,\mathfrak{t})-1$, and consequently $$d_{b}(\y ,\mathfrak{t})=-d_{b}(-\y ,\mathfrak{t})-1\leq -\frac{1}{2}\;.$$ 

Since the intersection form on $W$ is trivial, Theorem \ref{inequality} can also be applied for the pair $(-W,\y )$ to give the inequality $d_{b}(\y ,\mathfrak{t})\geq -\frac{1}{2}$. Both inequalities amount to the equality $$d_{b}(\y ,\mathfrak{t})=-\frac{1}{2}$$ for any torsion $\sp $ structure $\mathfrak{t}$ on $\y $ which is a restriction of a $\sp $ structure on $W$. 
\end{proof}

\begin{subsection}{Double plumbings inside $\CC P^{2}\# \CC P^{2}$}
\label{CP}
We consider double plumbings inside $X=\CC P^{2}\# \CC P^{2}$. Our question is whether a chosen pair of classes $\alpha ,\beta \in H_{2}(\CC P^{2}\# \CC P^{2})$ with $\alpha \cdot \beta =2$ can be represented by a configuration of two spheres with only two geometric intersections. We will find suitable classes $\alpha ,\beta $ and apply Proposition \ref{prop2}.  

Now $H_{2}(\CC P^{2}\# \CC P^{2})\cong \ZZ ^{2}$ has a standard basis $(e_{1},e_{2})$ with $e_{i}$ representing the class of the cycle $\CC P^{1}\subset \CC P^{2}$. The intersection form $Q_{X}$ of the manifold $X$ is given by 
$\left( \begin{array}{cc}
1 & 0 \\
0 & 1
\end{array} \right )$ and $b_{2}^{-}(X)=0$. We need to choose homologically independent classes $\alpha ,\beta \in H_{2}(X)$ that are both representable by spheres and for which $\alpha \cdot \beta =2$. A class $\zeta =(a,b)\in H_{2}(X)$ has a smooth representative $\Sigma $ of genus $$g(\Sigma )=\frac{(|a|-1)(|a|-2)}{2}+\frac{(|b|-1)(|b|-2)}{2}\;.$$ This representative is obtained by the connected sum of minimal genus representatives for classes of divisibility $a$ and $b$ in $\CC P^{2}$. Thus, nontrivial classes with smooth representatives of genus 0 are given by $ae_{1}+be_{2}\in H_{2}(X)$ where $(|a|,|b|)\in \{0,1,2\}^{2}\backslash \{(0,0)\}$. Up to isomorphism, there are three possible cases for $\alpha $ and $\beta $:   
\begin{xalignat*}{1}
& 2e_{1}+2e_{2}\textrm{ and }2e_{1}-e_{2}\\
& 2e_{1}\textrm{ and }e_{1}+2e_{2}\\
& e_{1}\textrm{ and }2e_{1}+e_{2}
\end{xalignat*}
We will investigate two cases: $\alpha =2e_{1}+2e_{2}, \beta =2e_{1}-e_{2}$ and $\alpha =2e_{1}, \beta =e_{1}+2e_{2}$. For the final case $\alpha =e_{1}$ and $\beta =2e_{1}\pm e_{2}$, the two classes can be represented by a pair of spheres intersecting in two points.  

\begin{subsubsection}{First case: $\alpha =2e_{1}+2e_{2}, \beta =2e_{1}-e_{2}$}
\label{FirstCP}
We have $m=\alpha ^{2}=8, n=\beta ^{2}=5$ and $$H_{1}(\y )=\ZZ \oplus \ZZ _{36}\;.$$ 
We will prove here the first part of Theorem \ref{app}, which says that any two spheres representing the classes $\alpha $ and $\beta $ intersect with at least 4 geometric intersections, and that there exist representatives with exactly 4 intersections.
\begin{proof}[Proof of Theorem \ref{app} a)]
Suppose there are spheres representing $\alpha $ and $\beta $ which have only two geometric intersections. Then the regular neighbourhood of their union is the double plumbing $N_{8,5}$ with boundary $Y_{8,5}$. Applying Theorem \ref{th1} we calculate the bottom-most correction terms $d_{b}$ in all torsion $\sp $ structures on $Y_{8,5}$:
\begin{center}
\begin{tabular}{|l|l|}
\hline
$\sp $\textrm{structure} & $d_{b}(\y ,\mathfrak{s}_{i,j})$\\ \hline
$\mathfrak{s}_{4,3}$ & $-17/18$\\ \hline
$\mathfrak{s}_{3,3}$ & $-3/4$\\ \hline
$\mathfrak{s}_{2,3}$ & $-5/18$\\ \hline
$\mathfrak{s}_{1,3}$ & $17/36$\\ \hline
$\mathfrak{s}_{6,0}$ & $3/2$\\ \hline
$\mathfrak{s}_{5,0}$ & $29/36$\\ \hline
$\mathfrak{s}_{4,0}$ & $7/18$\\ \hline
$\mathfrak{s}_{3,0}$ & $1/4$\\ \hline
$\mathfrak{s}_{2,0}$ & $7/18$\\ \hline
$\mathfrak{s}_{1,0}$ & $29/36$\\ \hline
$\mathfrak{s}_{0,0}$ & $3/2$\\ \hline
$\mathfrak{s}_{7,2}$ & $17/36$\\ \hline
$\mathfrak{s}_{6,2}$ & $-5/18$\\ \hline
$\mathfrak{s}_{5,2}$ & $-3/4$\\ \hline
$\mathfrak{s}_{4,2}$ & $-17/18$\\ \hline
$\mathfrak{s}_{3,2}$ & $-31/36$\\ \hline
$\mathfrak{s}_{2,2}$ & $-1/2$\\ \hline
$\mathfrak{s}_{1,2}$ & $5/36$\\ \hline
\end{tabular}
\begin{tabular}{|l|l|}
\hline
$\sp $\textrm{structure} & $d_{b}(\y ,\mathfrak{s}_{i,j})$\\ \hline
$\mathfrak{s}_{0,2}$ & $19/18$\\ \hline
$\mathfrak{s}_{7,4}$ & $1/4$\\ \hline
$\mathfrak{s}_{6,4}$ & $-5/18$\\ \hline
$\mathfrak{s}_{5,4}$ & $-19/36$\\ \hline
$\mathfrak{s}_{4,4}$ & $-1/2$\\ \hline
$\mathfrak{s}_{3,4}$ & $-7/36$\\ \hline
$\mathfrak{s}_{2,4}$ & $7/18$\\ \hline
$\mathfrak{s}_{1,4}$ & $5/4$\\ \hline
$\mathfrak{s}_{6,1}$ & $7/18$\\ \hline
$\mathfrak{s}_{5,1}$ & $-7/36$\\ \hline
$\mathfrak{s}_{4,1}$ & $-1/2$\\ \hline
$\mathfrak{s}_{3,1}$ & $-19/36$\\ \hline
$\mathfrak{s}_{2,1}$ & $-5/18$\\ \hline
$\mathfrak{s}_{1,1}$ & $1/4$\\ \hline
$\mathfrak{s}_{0,1}$ & $19/18$\\ \hline
$\mathfrak{s}_{7,3}$ & $5/36$\\ \hline
$\mathfrak{s}_{6,3}$ & $-1/2$\\ \hline
$\mathfrak{s}_{5,3}$ & $-31/36$\\ \hline
\end{tabular}
\end{center}

There are only four $\sp $ structures on $Y_{8,5}$ for which the equality $d_{b}(Y_{8,5},\mathfrak{t})=-\frac{1}{2}$ is valid, namely $\mathfrak{s}_{2,2}$, $\mathfrak{s}_{4,4}$, $\mathfrak{s}_{4,1}$ and $\mathfrak{s}_{6,3}$. It follows from Proposition \ref{prop2} that the two spheres which represent the  classes $\alpha ,\beta \in H_{2}(\CC P^{2}\# \CC P^{2})$ have to intersect with a geometric intersection number greater than 2. 
\begin{figure}[here]
\labellist
\footnotesize \hair 2pt
\pinlabel $2e_{1}+2e_{2}$ at 420 358
\pinlabel $2e_{1}-e_{2}$ at 422 164
\endlabellist
\begin{center}
\includegraphics[scale=0.50]{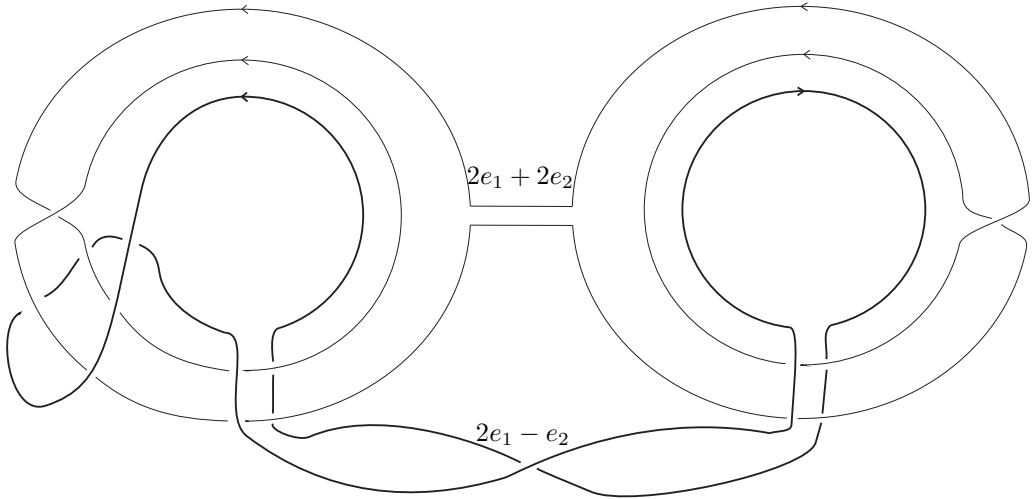}
\caption{Attaching circles of the 2-handles representing classes $2e_{1}+2e_{2}, 2e_{1}-e_{2}\in \CC P^{2}\# \CC P^{2}$ with four geometric intersections}
\label{fig:CP2}
\end{center}
\end{figure}

It is possible to construct genus zero representatives for $\alpha $ and $\beta $ with 4 geometric intersections. We use the following construction of Ruberman \cite{DR}: we represent $\CC P^{2}\# \CC P^{2}$ as a handlebody with two 2-handles with framing 1 and denote by $h_{1}$ and $h_{2}$ the cores of the 2-handles. By adding to $h_{i}$ a disk its boundary spans in $B^{4}$, we obtain a sphere representing $e_{i}$. Now let us represent the class $\alpha =2e_{1}+2e_{2}$: first we take two copies of $h_{i}$ and resolve their double point to get a single disk for $i=1,2$. Then we make a boundary connected sum of both disks (with coherent orientations) and add a disk in $B^{4}$ to the resulting surface. Similarly, we represent the class $\beta =2e_{1}-e_{2}$: first we take two copies of $h_{1}$ and resolve their double point, then we boundary connect sum the obtained disk and $h_{2}$ with the reversed orientation (this means the connected sum is made via a band with a half-twist) and add a disk in $B^{4}$ in the end. In this way we get the two spheres representing classes $\alpha $ and $\beta $ in $\CC P^{2}\# \CC P^{2}$. Figure \ref{fig:CP2} shows the two representatives in $\CC P^{2}\# \CC P^{2}$. The right loop of the dark curve can be slightly pulled left by an isotopy so that it intersects the light curve only twice, thus there remain only four intersections between the two spheres. It follows that 4 is the minimal number of geometric intersections.  
\end{proof}

\end{subsubsection} 
\begin{subsubsection}{Second case: $\alpha =2e_{1}, \beta =e_{1}+2e_{2}$}
\begin{figure}[here]
\labellist
\footnotesize \hair 2pt
\pinlabel $2e_{1}$ at 180 420
\pinlabel $e_{1}+2e_{2}$ at 438 420
\endlabellist
\begin{center}
\caption{Attaching circles of the 2-handles representing classes $2e_{1}, e_{1}+2e_{2}\in \CC P^{2}\# \CC P^{2}$ with two geometric intersections}
\includegraphics[scale=0.50]{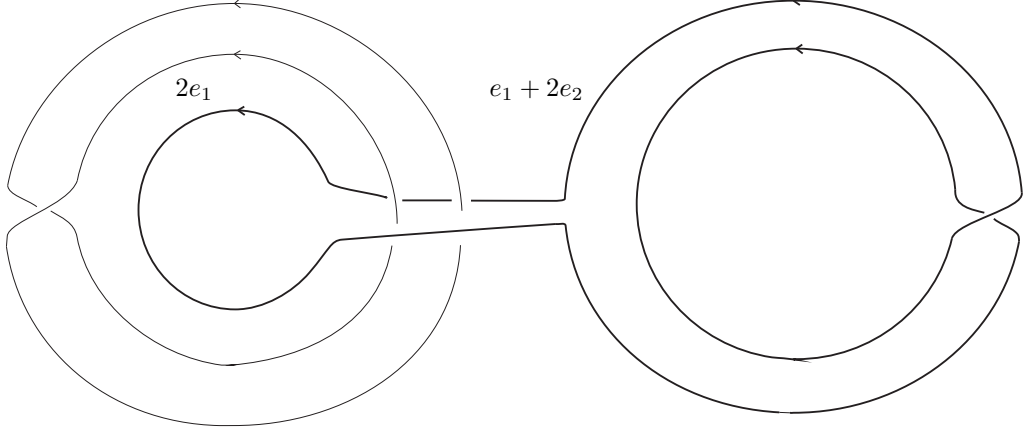}
\label{fig:CP3}
\end{center}
\end{figure}
The squares $m=\alpha ^{2}=4$ and $n=\beta ^{2}=5$ imply that $H_{1}(Y_{4,5})=\ZZ \oplus \ZZ _{16}$. The bottom-most correction terms $d_{b}$ of $Y_{4,5}$ are given by
\begin{center}
\begin{tabular}{|l|l|}
\hline
$\sp $\textrm{structure} & $d_{b}(\y ,\mathfrak{s}_{i,j})$\\ \hline
$\mathfrak{s}_{2,3}$ & $-15/16$\\ \hline
$\mathfrak{s}_{1,3}$ & $-1/2$\\ \hline
$\mathfrak{s}_{2,0}$ & $9/16$\\ \hline
$\mathfrak{s}_{1,0}$ & $1/4$\\ \hline
$\mathfrak{s}_{0,0}$ & $9/16$\\ \hline
$\mathfrak{s}_{3,2}$ & $-1/2$\\ \hline
$\mathfrak{s}_{2,2}$ & $-15/16$\\ \hline
$\mathfrak{s}_{1,2}$ & $-3/4$\\ \hline
\end{tabular}
\begin{tabular}{|l|l|}
\hline
$\sp $\textrm{structure} & $d_{b}(\y ,\mathfrak{s}_{i,j})$\\ \hline
$\mathfrak{s}_{0,2}$ & $1/16$\\ \hline
$\mathfrak{s}_{3,4}$ & $-1/2$\\ \hline
$\mathfrak{s}_{2,4}$ & $-7/16$\\ \hline
$\mathfrak{s}_{1,4}$ & $1/4$\\ \hline
$\mathfrak{s}_{2,1}$ & $-7/16$\\ \hline
$\mathfrak{s}_{1,1}$ & $-1/2$\\ \hline
$\mathfrak{s}_{0,1}$ & $1/16$\\ \hline
$\mathfrak{s}_{3,3}$ & $-3/4$\\ \hline
\end{tabular}
\end{center}
There are the requisite four $\sp $ structures on $Y_{4,5}$ for which $d_{b}$ is equal to $-\frac{1}{2}$:   
$$d_{b}(Y_{4,5},\mathfrak{s}_{1,3})=d_{b}(Y_{4,5},\mathfrak{s}_{3,2})=d_{b}(Y_{4,5},\mathfrak{s}_{3,4})=d_{b}(Y_{4,5},\mathfrak{s}_{1,1})=-\frac{1}{2}\;.$$ 
Indeed, one can choose the two spheres representing classes $\alpha $ and $\beta $ so that their geometric intersection consists of two points, see Figure \ref{fig:CP3}. 
\end{subsubsection}
\end{subsection}

\begin{subsection}{Double plumbings inside $S^{2}\times S^{2}\# S^{2}\times S^{2}$}
\label{S2}
Let us investigate double plum\-bings inside the 4--manifold $X=S^{2}\times S^{2}\# S^{2}\times S^{2}$. Since $X$ is simply connected and its intersection form $Q_{X}=\left( \begin{array}{cccc}
0 & 1 & 0 & 0 \\
1 & 0 & 0 & 0\\
0 & 0 & 0 & 1\\
0 & 0 & 1 & 0\\
\end{array} \right )$ is even, it follows that $X$ is a spin 4--manifold. According to \cite[Theorem 3]{WALL}, if $M$ is a simply connected closed oriented 4--manifold with an indefinite intersection form, then every primitive non\-cha\-rac\-teris\-tic class of $H_{2}(M\# (S^{2}\times S^{2}))$ is re\-pre\-sen\-ted by an embedded sphere. More specifically, Hirai showed that every primitive element of $H_{2}(S^{2}\times S^{2}\# S^{2}\times S^{2})$ can be represented by a smoothly embedded sphere \cite[Theorem 1]{HI}. A class $\mathbf{r}\in H_{2}(X)$ is pri\-mi\-tive if it cannot be written as $d\mathbf{t}$ for any class $\mathbf{t}\in H_{2}(X)$ and any $d\in \ZZ \backslash \{-1,1\}$.  

Denote by $(e_{1},e_{2},e_{3},e_{4})$ the standard basis of $H_{2}(S^{2}\times S^{2}\# S^{2}\times S^{2})$ and consider the classes $$\alpha =ae_{1}+2e_{2},\qquad \beta =e_{1}+te_{3}+e_{4}$$ where $a,t\in \mathbb{N}$ and $a$ is an odd number. We have $m=\alpha ^{2}=4a$, $n=\beta ^{2}=2t$ and $\alpha \cdot \beta =2$. Since $a$ is odd, the classes $\alpha $ and $\beta $ can be represented by spheres. We will prove the second part of Theorem \ref{app}, which says that if $a\geq 5$, then the spheres representing $\alpha $ and $\beta $ intersect with at least 4 geometric intersections. 
\begin{proof}[Proof of Theorem \ref{app} b)]
Suppose these two spheres have exactly two geometric intersections. We denote by $\n $ the regular neighborhood of the union of the spheres and by $W$ its complementary submanifold $W=X\backslash \operatorname{Int}(\n )$ in $X$. While $\n $ is the double plumbing of two disk bundles over spheres whose intersection form is positive definite, the submanifold $W\subset X$ carries the negative part of the intersection form. We have defined $\sp $ structures $\mathfrak{t}_{i,j}$ on $\n $ and denoted by $\mathfrak{s}_{i,j}=\mathfrak{t}_{i,j}|_{\y }$ the restriction of each $\sp $ structure to the boundary 3-manifold. Now we would like to define a $\sp $ structure $\mathfrak{u}_{i,j}\in \operatorname{Spin}^{c}(X)$ for which $\mathfrak{u}_{i,j}|_{\n }=\mathfrak{t}_{i,j}$. Then we will find the restriction $\mathfrak{u}_{i,j}|_{W}$ and use Theorem \ref{inequality} for the pair $(W,-\y )$, equipped with the $\sp $ structure $\mathfrak{u}_{i,j}|_{W}$ for some $i$ and $j$. 
By definition of $\mathfrak{t}_{i,j}\in \operatorname{Spin}^{c}(\n )$, we have $\langle c_{1}(\mathfrak{t}_{i,j}),\alpha \rangle =2i-m$ and $\langle c_{1}(\mathfrak{t}_{i,j}),\beta \rangle =2j-n$. 
For an odd $i$, define a $\sp $ structure $\mathfrak{u}_{i,j}$ on $X$ by 
\begin{align*}
& \langle c_{1}(\mathfrak{u}_{i,j}),e_{1}\rangle = \langle c_{1}(\mathfrak{u}_{i,j}),e_{3} \rangle=-2\\
& \langle c_{1}(\mathfrak{u}_{i,j}),e_{2} \rangle =i-a\\
& \langle c_{1}(\mathfrak{u}_{i,j}),e_{4} \rangle =2j+2
\end{align*}
Then we have $\langle c_{1}(\mathfrak{u}_{i,j}),\alpha \rangle =2i-m$ and $\langle c_{1}(\mathfrak{u}_{i,j}),\beta \rangle =2j-n$, which means that  $\mathfrak{u}_{i,j}|_{\n }=\mathfrak{t}_{i,j}$ and consequently $\mathfrak{u}_{i,j}|_{\y }=\mathfrak{s}_{i,j}$. We can calculate that the orthogonal complement of $H_{2}(\n )$ in $H_{2}(X)$ is spanned by the vectors $\gamma =-ae_{1}+2e_{2}-2e_{3}$ and $\delta =-te_{3}+e_{4}$, for which we have $\gamma ^{2}=-m$, $\delta ^{2}=-n$ and $\gamma \cdot \delta =-2$. Thus, $\gamma $ and $\delta $ are generators of $H_{2}(W)$ and its intersection form is given by the matrix $Q_{W}=\left( \begin{array}{cc}
-m & -2 \\
-2 & -n \\
\end{array} \right )$\;. We calculate 
\begin{align*}
& \langle c_{1}(\mathfrak{u}_{i,j}|_{W}),\gamma \rangle =2i+4\\
& \langle c_{1}(\mathfrak{u}_{i,j}|_{W}),\delta \rangle =n+2j+2
\end{align*}
It follows that the square of the first Chern class $c_{1}(\mathfrak{u}_{i,j}|_{W})$ is given by
\begin{align*}
c_{1}(\mathfrak{u}_{i,j}|_{W})^{2}=-\frac{1}{mn-4}\left (n(2i+4)^{2}+m(n+2j+2)^{2}-4(2i+4)(n+2j+2)\right )\;.
\end{align*}
Now the restriction of $\mathfrak{u}_{i,j}|_{W}$ to the boundary $-\y $ is the $\sp $ structure $\mathfrak{s}_{i,j}$ and Theorem \ref{inequality} implies  $4d_{b}(-\y ,\mathfrak{s}_{i,j})\geq c_{1}(\mathfrak{u}_{i,j}|_{W})^{2}$. Recall from Theorem \ref{th1} the correction terms $d_{b}(\y )$ and compare
\begin{align*}
& -d_{b}(\y ,\mathfrak{s}_{i,j})-1=d_{b}(-\y ,\mathfrak{s}_{i,j})\geq \frac{c_{1}(\mathfrak{u}_{i,j}|_{W})^{2}}{4}\\
& -\frac{c_{1}(\mathfrak{u}_{i,j}|_{W})^{2}}{4}\geq d_{b}(\y ,\mathfrak{s}_{i,j})+1\\
& \frac{n(2i+4)^{2}+m(n+2j+2)^{2}-4(2i+4)(n+2j+2)}{4(mn-4)}\geq \\
& \geq \frac{m^{2}n+mn^{2}-4mn(i+j)+4n(i^{2}+2i)+4m(j^{2}+2j)-16ij-16}{4(mn-4)}
\end{align*}
By simplifying this expression we get the inequality
\begin{align*}
& 4(mn-4)(i+2j+1-a)\geq 0\\
& i+2j+1\geq a
\end{align*}
where $1\leq i\leq 4a-1$ and $1\leq j\leq 2t-1$ and $i$ is odd. If $a\geq 5$, this inequality does not hold for the $\sp $ structure $\mathfrak{s}_{1,1}$. The higher the value of $a$, the more $\sp $ structures $\mathfrak{s}_{i,j}$ do not satisfy the above inequality. Therefore the two spheres representing $\alpha $ and $\beta $ must have at least 4 geometric intersections for all $a\geq 5$. 
\end{proof} 

It might be interesting to compare our result with \cite[Proposition 3.6]{ASKI}. According to the Proposition in the case $n=2$, the classes $(p_{1},q_{1},0,0)$ and $(0,0,p_{2},q_{2})$ (where $p_{i},q_{i}\geq 2$ and $(p_{i},q_{i})=1$ for $i=1,2$) are not disjointly, smoothly, $S^{2}$-representable inside the manifold $S^{2}\times S^{2}\# S^{2}\times S^{2}$.  

The application of $d_{b}$-invariants in the Section \ref{App} is similar to the $d$-invariant obstruction that is used for concordance applications, e.g. in \cite{JANA} and many other papers. 

\end{subsection}

\end{section}
\begin{section}{Geometric intersections of spheres with algebraic intersection one}
\label{One}
Now we investigate a configuration of two spheres which intersect only once inside a closed smooth 4--manifold $X$ with $H_{1}(X)=0$ and $b_{2}^{+}(X)=2$. Such a configuration is a (single) plumbing $\m $ of disk bundles over spheres with Euler numbers $m$ and $n$. The Kirby diagram for $\m $ is a Hopf link of two framed unknoted circles, which can be changed by the operation called slam-dunk \cite[page 163]{kirby} into a single unknoted circle with framing $\frac{mn-1}{n}$. The boundary of $\m $ is thus the lens space $L(mn-1,n)$ with $H_{1}(L(mn-1,n))=\ZZ _{mn-1}$. For labeling lens spaces, we use notation from \cite{OS4}. By the results of \cite[Proposition 3.1]{OS2}, the Heegaard--Floer homology of $\widehat{HF}(L(p,q))$ has one generator in every torsion $\sp $ structure and its absolute grading is given by a recursive formula from \cite[Proposition 4.8]{OS4}: $$d(-L(p,q),i)=\left (\frac{pq-(2i+1-p-q)^{2}}{4pq}\right )-d(-L(q,r),j)$$ where $r$ and $j$ are the reductions of $p$ and $i$ modulo $q$ respectively. In our case $p=mn-1$ and $q=n$, so $r=n-1$. 
In the special case when $n=1$, we need only one application of the recursive formula to obtain
\begin{xalignat}{1}
\label{lensD1}
& d(-L(m-1,1),i)=-\frac{(m-2i-1)^{2}}{4(m-1)}+\frac{1}{4}\;.
\end{xalignat}
In another special case when $n=2$, we need two applications of the recursive formula to obtain
\begin{xalignat}{1}
\label{lensD2}
& d(-L(2m-1,2),i)=-\frac{(m-i)^{2}}{2(2m-1)}+\frac{1}{2}\quad \textrm{if $i$ is even,}\\
& d(-L(2m-1,2),i)=-\frac{(m-i)^{2}}{2(2m-1)}\quad \textrm{if $i$ is odd.}
\end{xalignat}
When $n>2$, starting with $d(-L(n-1,1),j)$ we apply the recursive formula three times to obtain \begin{xalignat*}{1}
& d(-L(mn-1,n),i)=\frac{1}{4}-\frac{(2i+2-mn-n)^{2}}{4n(mn-1)}+\frac{(2j+2-2n)^{2}}{4n(n-1)}-\frac{(2t+1-n)^{2}}{4(n-1)}
\end{xalignat*} where $j$ is the reduction of $i\textrm{ mod $n$}$ and $t$ is the reduction of $j\textrm{ mod ($n-1$)}$. In the special case when $0\leq i<n-1$ and thus $i=j=t$ we get a simplification 
\begin{xalignat}{1}
\label{lensD}
& d(-L(mn-1,n),i)=-\frac{1}{4(mn-1)}\left (nm^{2}+m(n-2i)^{2}-2m(n-2i)\right )+\frac{2}{4}
\end{xalignat}
Denote $L=-L(mn-1,n)$. Let us derive the formula \eqref{lensD} in another way: by defining a $\sp $ structure $\mathfrak{s}_{i}$ on the plumbing $-\m $ and using the Formula \eqref{grading} from \cite[Formula (4)]{OS4} to compute $d(L,\mathfrak{s}_{i}|_{L})$. By removing a 4-ball from $-\m $ we get a cobordism $\mathcal{C}$ from $S^{3}$ to $L$. Since the intersection form of $-\m $ is given by the matrix $Q_{-\m }=\left( \begin{array}{cc}
-m & -1 \\
-1 & -n \\
\end{array} \right )$, we have $\chi (\mathcal{C})=2$ and $\sigma (\mathcal{C})=-2$. Define a $\sp $ structure $\mathfrak{s}_{i}$ on $-\m $ by 
\begin{align}
\label{defs}
& \langle c_{1}(\mathfrak{s}_{i}),s_{1}\rangle =m,\quad \quad \quad \langle c_{1}(\mathfrak{s}_{i}),s_{2}\rangle =n-2i\;,
\end{align} where $s_{1},s_{2}\in H_{2}(-\m )$ are the classes of the base spheres in the plumbing $-\m $. It follows that $$c_{1}(\mathfrak{s}_{i})^{2}=-\frac{nm^{2}+m(n-2i)^{2}-2m(n-2i)}{mn-1}$$ and the formula \eqref{grading} gives us 
\begin{align*}
& d(L,\mathfrak{s}_{i}|_{L})=-\frac{nm^{2}+m(n-2i)^{2}-2m(n-2i)}{4(mn-1)}+\frac{2}{4}\;,
\end{align*} which coincides with Formula \eqref{lensD}. 

From Theorem \ref{inequality} we obtain the following obstruction for the $d$-invariants: 
\begin{proposition} 
\label{prop4}
Let $X$ be a closed smooth 4-manifold with $H_{1}(X)=0$, $b_{2}^{+}(X)=2$ and $b_{2}^{-}(X)=0$. Suppose there are two spheres $\Sigma _{1},\Sigma _{2}\subset X$ with $\Sigma _{1} ^{2}=m>0$, $\Sigma _{2}^{2}=n>0$ and $\Sigma _{1}\cdot \Sigma _{2}=1$. Denote by $L$ the boundary of a regular neighbourhood of $\Sigma _{1}\cup \Sigma _{2}$. Then for some subgroup $\tau \subset H_{1}(L)$ with $|\tau |^{2}=mn-1$ and some $\sp $ structure $\mathfrak{s}_{0}$ on $L$, we have $d(L,\mathfrak{s}_{0}+\phi )=0$ for every $\phi \in \tau$. 
\end{proposition}
\begin{proof}
Denote by $\m $ the regular neighbourhood of $\Sigma _{1}\cup \Sigma _{2}$ and let $V=X\backslash \operatorname{Int}(\m )$. We study the Mayer--Vietoris sequence in cohomology of the triple $(X,V,\m )$:
\begin{eqnarray*}
& 0\rightarrow H^{1}(V)\oplus H^{1}(\m )\stackrel{f_{1}}\rightarrow H^{1}(L) \stackrel{f_{2}}\rightarrow  H^{2}(X)\stackrel{f_{3}}\rightarrow H^{2}(V)\oplus H^{2}(\m ) \stackrel{f_{4}}\rightarrow H^{2}(L)\rightarrow 0\\
& 0\rightarrow H^{1}(V)\oplus H^{1}(\m )\stackrel{f_{1}}\rightarrow 0 \stackrel{f_{2}}\rightarrow \ZZ ^{2}\stackrel{f_{3}}\rightarrow H^{2}(V)\oplus H^{2}(\m )\stackrel{f_{4}}\rightarrow \ZZ _{mn-1}\rightarrow 0
\end{eqnarray*}  
At the beginning and at the end of the sequence we have zeroes since $H_{1}(X)=0$. Since $L$ is the lens space $L(mn-1,n)$, we have $H^{2}(L)=\ZZ _{mn-1}$ and $H^{1}(L)=0$. It follows from the sequence that $H^{1}(V)=H^{1}(\m )=0$, so $H_{1}(V)=\tau $ is a torsion group by the universal coefficient theorem. The group $H_{2}(\m )=\ZZ ^{2}$ is spanned by the homology classes of the spheres $\Sigma _{1}$ and $\Sigma _{2}$, so the cohomology group $H^{2}(\m )$ has rank two. It follows that $H^{2}(V)\cong \tau $ and $H_{2}(V)=0$. Now we can write down the homology long exact sequence of the pair $(V,-L)$:   
\begin{eqnarray*}
& \rightarrow H_{2}(-L)\stackrel{g_{1}}\longrightarrow H_{2}(V)\stackrel{g_{2}}\longrightarrow H_{2}(V,-L)\stackrel{g_{3}}\longrightarrow H_{1}(-L)\stackrel{g_{4}}\longrightarrow H_{1}(V)\rightarrow \ldots \\
& \rightarrow 0 \stackrel{g_{1}}\longrightarrow 0\stackrel{g_{2}}\longrightarrow \tau \stackrel{g_{3}}\longrightarrow \ZZ _{mn-1}\stackrel{g_{4}}\longrightarrow \tau \stackrel{0}\rightarrow \ldots 
\end{eqnarray*}
It follows from this sequence that $\tau $ is a subgroup of $\ZZ _{mn-1}$ with quotient group $\ZZ _{mn-1}/\tau \cong \tau $, thus $|\tau |^{2}=mn-1$. Those $\sp $ structures on $-L$ which are restrictions of $\sp $ structures on $V$ correspond to the image of the map $H^{2}(V)\to H^{2}(-L)$, which is the monomorphism $\tau \to \ZZ _{mn-1}$. For every $\sp $ structure on $-L$ which is the restriction of a $\sp $ structure on $V$ we can apply Theorem \ref{inequality} to obtain the estimate $d(-L,\mathfrak{s})\geq 0$ and consequently $d(L,\mathfrak{s})\leq 0$ . Since $V$ has a trivial intersection form, we can also apply the same theorem for the pair $(-V,L)$ to obtain $d(L,\mathfrak{s})\geq 0$, from which the equality follows. 
\end{proof}

\begin{subsection}{Single plumbings inside $\CC P^{2}\# \CC P^{2}$}
\label{CP1}
Let $X=\CC P^{2}\# \CC P^{2}$ and denote by $(e_{1},e_{2})$ the standard basis for $H_{2}(X)$. As remarked in Subsection \ref{CP}, the classes in $H_{2}(X)$ which are representable by spheres have the form $x_{1}e_{1}+x_{2}e_{2}$ with $(x_{1},x_{2})\in \{0,\pm 1,\pm 2\}^{2}\backslash \{(0,0)\}$. Consider a pair of such classes with algebraic intersection 1: $\alpha =2e_{1}+e_{2}$ and $\beta =e_{1}-e_{2}$. We have $m=\alpha ^{2}=5$ and $n=\beta ^{2}=2$ so $L=L(9,2)$ and the $d$--invariants are given by
\begin{xalignat*}{1}
& d(L(9,2),i)=\frac{(i-5)^{2}}{18}-\frac{(j-1)^{2}}{2}
\end{xalignat*}
for $0\leq i\leq 8$, where $j$ is the reduction of $i$ (mod $2$). We calculate 
\begin{xalignat*}{1}
& d(L,0)=d(L,1)=\frac{8}{9}\qquad d(L,2)=d(L,5)=d(L,8)=0\\
& d(L,3)=d(L,7)=\frac{2}{9}\qquad d(L,4)=d(L,6)=-\frac{4}{9}
\end{xalignat*}
\begin{figure}[here]
\labellist
\footnotesize \hair 2pt
\pinlabel $2e_{1}+e_{2}$ at 420 346
\pinlabel $e_{1}-e_{2}$ at 423 163
\endlabellist
\begin{center}
\includegraphics[scale=0.50]{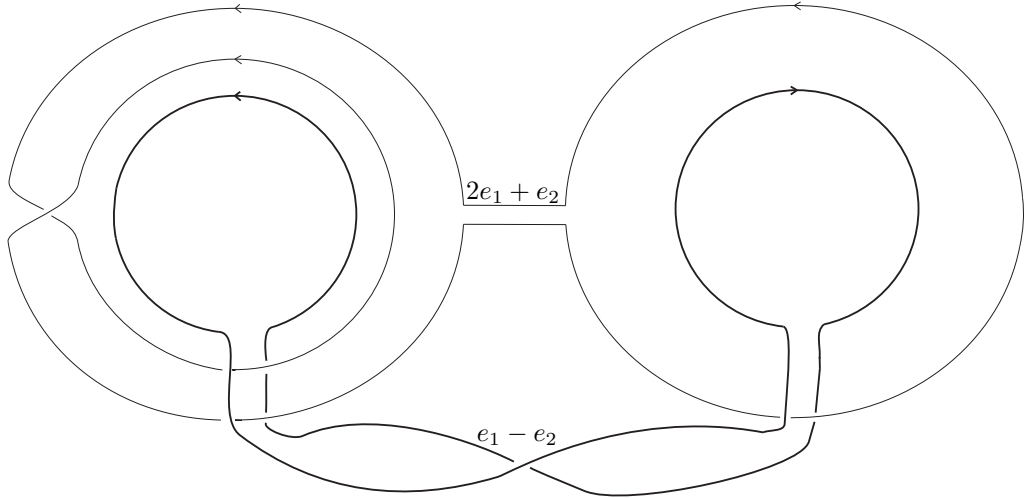}
\caption{Attaching circles of the 2-handles representing classes $2e_{1}+e_{2},e_{1}-e_{2}\in \CC P^{2}\# \CC P^{2}$}
\label{fig:CP1}
\end{center}
\end{figure}
We can see that there are three $\sp $ structures with $d$-invariant equal to 0, in accordance with Proposition \ref{prop4}. Thus, the spheres representing classes $\alpha $ and $\beta $ can have only one geometric intersection inside $\CC P^{2}\# \CC P^{2}$. Indeed, the two spheres can be chosen in such a way, following the construction of Ruberman \cite{DR} described in Subsubsection \ref{FirstCP}. We represent the class $\alpha =2e_{1}+e_{2}$ by taking two copies of $h_{1}$, resolve their double point to get a single disk, then make a boundary connected sum with $h_{2}$ (with coherent orientations) and add a disk in $B^{4}$ to the resulting surface. Similarly, we represent the class $\beta =e_{1}-e_{2}$ by taking a boundary connected sum of $h_{1}$ and $h_{2}$ with reversed orientations and adding a disk in $B^{4}$, see Figure \ref{fig:CP1}. The attaching circles of the 2-handles thus achieved can be moved by an isotopy to form the Hopf link, which shows that the two representatives have only one geometric intersection.
    
\end{subsection}

\begin{subsection}{Single plumbings inside $S^{2}\times S^{2}\# S^{2}\times S^{2}$}
\label{S21}

Consider the 4--manifold $X=S^{2}\times S^{2}\# S^{2}\times S^{2}$ and two classes $\alpha =(2k+1)e_{1}+2e_{2}$ and $\beta =-ke_{1}+e_{2}+2ke_{3}+e_{4}$ in $H_{2}(X)$, where $k$ is a positive integer. We have $\alpha ^{2}=4(2k+1)=m$, $\beta ^{2}=2k=n$ and $\alpha \cdot \beta =1$. Since $\alpha $ and $\beta $ are primitive noncharacteristic classes, they are represented by embedded spheres in $X$ by \cite[Theorem 3]{WALL}. We will prove here Theorem \ref{app1} which says: Any two spheres representing the classes $\alpha $ and $\beta $ intersect with at least 3 geometric intersections for all $k>1$. 

\begin{proof}[Proof of Theorem \ref{app1}]
Suppose the two spheres intersect with only one geometric intersection; then a regular neighborhood of their configuration forms the plumbing $\m $ inside $X$. Denote by $V=X\backslash \operatorname{Int}(\m )$ its complementary submanifold and let $L=\partial V$.  

We would like to define a $\sp $ structure $\mathfrak{t}_{i}$ on $X$, for which the restriction $\mathfrak{t}_{i}|_{\m }=\mathfrak{s}_{i}$. Then we will find the restriction $\mathfrak{t}_{i}|_{V}$ to the complementary sub\-ma\-ni\-fold and apply Theorem \ref{inequality}. Let $\mathfrak{t}_{i}\in \operatorname{Spin}^{c}(X)$ be the unique $\sp $ structure for which the following holds:
\begin{align*}
& \langle c_{1}(\mathfrak{t}_{i}),e_{1}\rangle =0\\
& \langle c_{1}(\mathfrak{t}_{i}),e_{2}\rangle =2(2k+1)\\
& \langle c_{1}(\mathfrak{t}_{i}),e_{3}\rangle =-2\\
& \langle c_{1}(\mathfrak{t}_{i}),e_{4}\rangle =2(k-i-1)
\end{align*}
Then we have $\langle c_{1}(\mathfrak{t}_{i}),\alpha \rangle =4(2k+1)=m$ and $\langle c_{1}(\mathfrak{t}_{i}),\beta \rangle =2k-2i=n-2i$, which means that $\mathfrak{t}_{i}|_{\m }$ concides with the $\sp $ structure $\mathfrak{s}_{i}$ defined in \eqref{defs}. As we have shown,  the correction term $d(L,\mathfrak{s}_{i}|_{L})$ is given by the Formula \eqref{lensD}.  
Now let us find the restriction $\mathfrak{t}_{i}|_{V}$. The image of the inclusion homomorphism $H_{2}(V)\to H_{2}(X)$ is spanned by the two classes $\gamma =-(2k+1)e_{1}+2e_{2}+(4k+1)e_{3}$ and $\delta =-2ke_{3}+e_{4}$ which are both orthogonal to $\alpha $ and $\beta $. We calculate
\begin{align*}
& \langle c_{1}(\mathfrak{t}_{i}),\gamma \rangle =2\\
& \langle c_{1}(\mathfrak{t}_{i}),\delta \rangle =6k-2i-2=3n-2i-2
\end{align*}
Since $\gamma ^{2}=-m$, $\delta ^{2}=-2n$ and $\gamma \cdot \delta =4k+1=\frac{m-2}{2}$, the intersection form on $V$ is given by the matrix $Q_{V}=\left( \begin{array}{cc}
-m & \frac{m-2}{2} \\
\frac{m-2}{2} & -2n \\
\end{array} \right )$ with $\operatorname{det}Q_{V}=mn-1$. The square of $c_{1}(\mathfrak{t}_{i}|_{V})$ is then calculated by 
\begin{xalignat*}{1}
& c_{1}(\mathfrak{t}_{i}|_{V})^{2}=-\frac{8n+m(3n-2i-2)^{2}+2(m-2)(3n-2i-2)}{mn-1}\;.
\end{xalignat*} Now Theorem \ref{inequality} gives us the inequality $c_{1}(\mathfrak{t}_{i}|_{V})^{2}+2\leq 4d(L,i)$. Using the Equation \eqref{lensD}, we compare
\begin{align*}
 -\frac{8n+m(3n-2i-2)^{2}+2(m-2)(3n-2i-2)}{mn-1}\leq -\frac{nm^{2}+m(n-2i)^{2}-2m(n-2i)}{mn-1}
\end{align*} and by simplifying we get the inequality $$(mn-1)(k-i-1)\geq 0\;,$$ which is not valid for $i\geq k$. When applying the Formula \eqref{lensD} we assumed that $0\leq i<n-1=2k-1$. Thus, the $\sp $ structure $\mathfrak{s}_{k}|_{L}$ does not satisfy the inequality in Theorem \ref{inequality} whenever $k>1$. Therefore, the two spheres representing the classes $\alpha $ and $\beta $ have at least three geometric intersections for all $k>1$. 
\end{proof}

\end{subsection}
\end{section}

\bibliographystyle{plain}
\bibliography{lit1}

\end{document}